\documentclass[10pt]{elsarticle}
\journal{arXiv}
\usepackage{amssymb}
 \usepackage{amsthm}
 \usepackage{lineno}
 \usepackage{amsmath}
   \numberwithin{equation}{section}
\usepackage{mathrsfs}

\newtheorem{theorem}{Theorem}

\newtheorem{lemma}[theorem]{Lemma}
\newtheorem{proposition}[theorem]{Proposition}
\newtheorem{definition}[theorem]{Definition}
\newtheorem{remark}[theorem]{Remark}
\numberwithin{equation}{section} \numberwithin{theorem}{section}

\begin{document}
\begin{frontmatter}
\author{Qiang Gao$^{a}$, Xiaoyan Zhang$^{a,}$\footnote{Corresponding Author. Email: zxysd@sdu.edu.cn} \\
{\small $^a$ School of Mathematics, Shandong University, Jinan 250100, P.R. China}
}
\date{}

\title{ Ground state and multiple normalized solutions of quasilinear Schr\"odinger equations in the $L^2$-supercritical case and the Sobolev critical case}
\date{}

\begin{abstract}
This paper is devoted to studying the existence of normalized solutions for the following quasilinear Schr\"odinger equation
\begin{equation*}
\begin{aligned}
   -\Delta u-u\Delta u^2 +\lambda u=|u|^{p-2}u   \quad\mathrm{in}\ \mathbb{R}^{N},
\end{aligned}
\end{equation*}
where $N=3,4$, $\lambda$ appears as a Lagrange multiplier and $p \in (4+\frac{4}{N},2\cdot2^*]$. The solutions correspond to critical points of the energy functional subject to  the $L^2$-norm constraint $\int_{\mathbb{R}^N}|u|^2dx=a^2>0$. In the Sobolev critical case $p=2\cdot 2^*$, the energy functional has no critical point.
As for $L^2$-supercritical case $p \in (4+\frac{4}{N},2\cdot2^*)$: on the one hand, taking into account Pohozaev manifold and
perturbation method, we obtain the existence of ground state normalized solutions for the non-radial case; on the other hand, we get the existence of infinitely many normalized solutions in $H^1_r(\mathbb{R}^N)$. Moreover, our results cover several relevant existing results. And in the end, we get the asymptotic properties of energy as $a$ tends to $+\infty$ and $a$ tends to $0^+$.
\end{abstract}

\begin{keyword}
Quasilinear Schr\"odinger equations\sep Ground state, Multiple  solutions\sep Normalized solutions
\end{keyword}
\end{frontmatter}
{\bf Mathematics Subject Classification:} {\small 30C70, 35J10, 35J20, 35Q55}

\section{Introduction}\label{intro}
In this paper, we are concerned with the following quasilinear Schr\"odinger equation
\begin{equation}\label{eq01}
 -\Delta u-u\Delta u^2 +\lambda u=|u|^{p-2}u   \quad\mathrm{in}\ \mathbb{R}^{N},
\end{equation}
under the constraint
\begin{equation}\label{eq02}
\int_{\mathbb{R}^N}|u|^2dx=a^2,
\end{equation}
where $N=3,4$, $a>0$, $\lambda\in \mathbb{R}$ is a Lagrange multiplier and $p \in (4+\frac{4}{N},2\cdot2^*]$.

The interest in studying \eqref{eq01}-\eqref{eq02} comes from seeking the standing wave of the following time-dependent quasilinear Schr\"odinger equation
\begin{equation}\label{eq03}
i\partial_{t}\psi=\Delta\psi+\kappa \psi \Delta |\psi|^2+\sigma h(|\psi|)\psi,
\end{equation}
where $i$ denotes the imaginary unit, $\psi=\psi(t,x) \in \mathbb{C}$ is the wave function, $h$ is an approprite nonlinearity. Equations of the form \eqref{eq03} have been involved in models of superfluid films in fluid mechanics and plasma physics \cite{K1981, LS1978, PG1976}. And the equations play an important role in dissipative quantum mechanics, condensed matter theory, the theory of Heisenberg ferromagnets and magnons. As for more details on physical background of \eqref{eq03}, we refer the readers to \cite{BN1990,H1980,KIK1990,MF1984,QC1982}.

Throughout the paper, we consider standing  wave solutions to \eqref{eq03}, which are solutions of the form
$$\psi(t,x)=e^{-i\lambda t}u(x),$$
where $u \in H^1(\mathbb{R}^N)$, the frequency $\lambda \in \mathbb{R}$, then \eqref{eq03} can be transformed to \eqref{eq01} with $h(|u|)u=|u|^{p-2}u$.

Generally speaking, there are two  different points  concerning the form of standing wave solutions: one can either prescribe the frequency $\lambda$ or the $L^2$-norm $\|u\|_2$. In contrast with the search of solutions to \eqref{eq03} when the frequency $\lambda \in \mathbb{R}$ is a prior given, the search of normalized solutions becomes more complex. In the normalized setting, it is natural to prescribe the value of the mass $\int_{\mathbb{R}^N}|u|^2dx$ so that $\lambda$ can be interpreted as the Lagrange multiplier. Naturally, a new critical exponent appears, the $L^2$-critical exponent (also named mass-critical exponent): $p:=4+\frac{4}{N}$. It is the threshold exponent for many dynamical properties, such as global existence vs. blow-up, and the stability or instability of ground states. For further clarification, we agree that $L^2$-subcritical case and $L^2$-supercritical case mean that $p<4+\frac{4}{N}$ and $p>4+\frac{4}{N}$, respectively. Alternatively, the mass often admits a clear physical meaning: it represents the power supply in nonlinear optics, or the total number of atoms in Bose-Einstein condensation. They are two main fields of application of the NLS.

It is well known that the solutions for \eqref{eq01} admitting prescribed $L^2$-norm \eqref{eq02} are the critical points of the energy functional
$$I(u):=\frac{1}{2} \int_{\mathbb{R}^N} |\nabla u|^2 + \int_{\mathbb{R}^N} |u|^2 |\nabla u|^2 -\frac{1}{p}\int_{\mathbb{R}^N} |u|^{p} $$
restricted on $\mathcal{S}'(a):=\{u \in H^1(\mathbb{R}^N)\;| \;\int_{\mathbb{R}^N}|u|^2|\nabla u|^2 <+\infty, \int_{\mathbb{R}^N}|u|^2=a^2\}$. At this time, the frequency $\lambda$ is an unknown number that can be determined as the Lagrange multiplier associated to the constraint $\mathcal{S}'(a)$. In addition, it is quite meaningful to study normalized solutions. This is not only because mass is conserved along the trajectories of \eqref{eq03}, i.e.,
$$\int_{\mathbb{R}^N}|\psi(t,x)|^2 dx= \int_{\mathbb{R}^N}|\psi(0,x)|^2 dx$$
for all $t>0$, but also it can provide a good insight of the dynamical properties (such as, orbital stability and instability) of solutions to the equation \eqref{eq03}.

 We introduce some results about the existence of normalized solution. Jeanjean \cite{J1997} considered the semilinear Schr\"odinger equation
\begin{equation}\label{eq1.4}
-\triangle u +\lambda u =h(u),\;\;\; x \in \mathbb{R}^N,
\end{equation}
where $N \geq 1$, $\lambda \in \mathbb{R}$. By using of a minimax procedure, Jeanjean showed that for each $a>0$, \eqref{eq1.4} possesses at least one couple $(u_a,\lambda_a)$ of weak
solution with $\|u_a\|_2=a$ for $N\geq 2$. At the same time, he obtained the existence of ground states for $N\geq 1$. But, afterwards, there was little progress about the study of normalized solutions for
a long time. One of the main reasons is that it is hard to prove the boundedness of
constrained Palais-Smale sequence when the functional is unbounded from below on
the constraint manifold. More recently, problems of such type begun to receive much
attention. By virtue of a fountain theorem type argument, Bartsch and de Valeriola \cite{Bd2013} got a multiplicity result of \eqref{eq1.4} with $\|u\|_2=a>0$. Soave \cite{S2020} studied the existence and
properties of ground states to the nonlinear Schr\"odinger equation with combined
power nonlinearities. More results, we refer the readers to \cite{GZ2021, GY2024, LZ2023, LZ2024, ZhangZhitao, N2020, ZZ2022} and their references therein.

Compared to \eqref{eq1.4} where the term $u\Delta u^2$ is not present, the search of solutions of \eqref{eq01} and \eqref{eq02} presents a major difficulty. The functional associated with the quasilinear term $\int_{\mathbb{R}^N} |u|^2 |\nabla u|^2 $
is nondifferentiable in $\{u \in H^1(\mathbb{R}^N)\;| \;\int_{\mathbb{R}^N}|u|^2|\nabla u|^2 <+\infty \}$. Since the parameter $\lambda$ is unknown and $\|u\|_2$ is equal to a constant, we find that Nehari manifold approach \cite{LWW2004, LLW2013} and changing variables \cite{CJ2004, LWW2003} are no longer applicable. To overcome this difficulty, various arguments have been developed. One of the most important tasks is that Jeanjean, Luo and Wang \cite{JLW2015} introduced the perturbation method to deal with the normalized solutions of quasilinear Schr\"odinger equation.

As far as I know, there are relatively few results on \eqref{eq01} and \eqref{eq02}. In  \cite{CJS2010, JL2013}, the authors studied the minimization problem $\tilde{m}(a)=\inf_{u \in \mathcal{S}'(a)}\left\{\frac{1}{2} \int_{\mathbb{R}^N} |\nabla u|^2 + \int_{\mathbb{R}^N} |u|^2 |\nabla u|^2 -\frac{1}{p}\int_{\mathbb{R}^N} |u|^p\right\}$ with $2<p\leq 4+\frac{4}{N}$. In addition, the authors \cite{ZZ2018} considered the existence and asymptotic behavior of the minimizers to $\tilde{m}(a)=\inf_{u \in \mathcal{S}'(a)}\left\{\frac{1}{2} \int_{\mathbb{R}^N} \left(|\nabla u|^2 + V(x)|u|^2\right)+\int_{\mathbb{R}^N} |u|^2 |\nabla u|^2 -\frac{1}{p}\int_{\mathbb{R}^N} |u|^p\right\}$ with $2<p\leq 4+\frac{4}{N}$, where $V(x)$ is an infinite potential well. In a word, most of the results on normalized solution of \eqref{eq01} and \eqref{eq02} have been related to $L^2$-subcritical case and $L^2$-critical case. Specially, Li and Zou \cite{LZ2023} have considered  $L^2$-supcritical case with $p>4+\frac{4}{N}$, Mao and Lu \cite{2024New} have studied combination case of $L^2$-subcritical and $L^2$-supercritical
nonlinearities with $f(t)=\tau|t|^{q-2}t+|t|^{p-2}t$, $\tau>0$, $2<q<2+\frac{4}{N}$ and $p>4+\frac{4}{N}$. Recently, Gao and Guo \cite{GG2024} have proved the existence of a normalized solutions for  $L^2$-supcritical
quasilinear Schr\"odinger equation with potentials $V(x) \neq 0$.

First of all, we briefly consider the Sobolev critical case, and we can state the following non-existence result:

\begin{theorem}\label{Thm1}
Assume $N=3,4$ and $p=2\cdot2^*$, then $I(u)$ has no critical point on $\mathcal{S}'(a)$.
\end{theorem}

Next, we consider the $L^2$-supcritical case. On the one hand, in the non-radial space $H^1(\mathbb{R}^N)$, we have

\begin{theorem}\label{Thm1.1'}
Assume $N=3$ and $4+\frac{4}{N}<p\leq 2^*$, then equations \eqref{eq01} and \eqref{eq02} admit a ground state normalized solution $u \in H^1(\mathbb{R}^N)\cap L^{\infty}(\mathbb{R}^N) $.
\end{theorem}

\begin{remark}
In \cite{ZZJ2025} and \cite{ZCW2023},  the authors have shown that equations \eqref{eq01} and \eqref{eq02} admit a Schwarz symmetric ground state solution. In this case, the compactness naturally follows from the symmetry. However, in the absence of symmetry, as considered in this paper, compactness becomes more challenging to deal with. Since our approach relies on Lemma \ref{lem6.1'}, we are only able to establish the existence of a ground state for $4+\frac{4}{N}<p\leq 2^*$, rather than the broader range $4+\frac{4}{N}<p< 2\cdot2^*$. Lemma \ref{lem4.1'} can only be used to obtain the ground state solution; therefore, we are unable to derive multiple solutions.
\end{remark}

On the other hand, in the radial space $H^1_r(\mathbb{R}^N)$, our main result is

\begin{theorem}\label{Thm1.2}
Assume $N=3,4$ and $4+\frac{4}{N}<p<2\cdot2^*$, then equations \eqref{eq01} and \eqref{eq02} admit a sequence of radially symmetric normalized solutions $u^j \in H^1(\mathbb{R}^N)\cap L^{\infty}(\mathbb{R}^N) $ with increasing energy $I(u^j)\rightarrow +\infty$.
\end{theorem}

\begin{remark}
    Li and Zou \cite{LZ2023} have proved the existence of infinitely many radially symmetric solutions for $N=3$ and $4+\frac{4}{N}<p<2^*$. We used different methods to get better results. Theorem \ref{Thm1.2} not only include their results, but also generalize to $N=4$ and $4+\frac{4}{N}<p<2\cdot2^*$. Since Jeanjean, Zhang and Zhong \cite{ZZJ2025} have given a Schwarz symmetric ground state solution for $N=3,4$ and $4+\frac{4}{N}<p<2\cdot2^*$, we no longer deal with radially symmetric ground state solutions.
\end{remark}

It is easy to see that if $u$ is the critical point of $I |_{\mathcal{S}(a)}$, then $u$ satisfies the following Pohozaev identity
$$P (u):=\int_{\mathbb{R}^N} |\nabla u|^2 +
(N+2)\int_{\mathbb{R}^N} |u|^2 |\nabla u|^2 -\gamma_p \int_{\mathbb{R}^N} |u|^{p}.$$
We define
 $$m(a):=\inf_{u \in \mathcal{P}'(a)} I(u),$$
 $$\sigma(a):=\inf_{u \in \mathcal{S}'(a), \Theta(u,\lambda)=0} I(u),$$
  where $\mathcal{P}'(a):=\{u\in \mathcal{S}'(a)\ |\ P(u)=0 \}$, $\Theta(u,\lambda):=  -\Delta u-u\Delta u^2 -|u|^{p-2}u +\lambda u$. We can conclude $m(a)=\sigma(a)$ by \cite{ZZJ2025} or \cite{ZCW2023}.
 And we have the following results:
\begin{theorem}\label{Thm1.5}
   Assume $N \geq3$, for the function $a \mapsto m(a)$ we have

    (1) if $4+\frac{4}{N}<p<2\cdot2^*$, $m(a)$ is positive and lower semicontinuous;

    (2) if $4+\frac{4}{N}<p<2^*$, $m(a) \rightarrow 0^+$ as $a \rightarrow +\infty$;

    (3) if $4+\frac{4}{N}<p<2\cdot2^*$, $m(a) \rightarrow +\infty$ as $a \rightarrow 0^+$.
\end{theorem}


    We complete the introduction by sketching the structure of the paper. In Section 2, we define the perturbed functional, discuss the geometry of $I_\mu$ and give the properties of Pohozaev manifold. In Section 3, we give the proof of Theorem \ref{Thm1}. In Section 4, we deal with the case of non-radial space and get Theorem \ref{Thm1.1'}. In Section 5, we consider the radial space $H^1_r(\mathbb{R}^N)$ and finish the proof of Theorem \ref{Thm1.2}. In Section 6, we study the lower semi-continuity of  $m(a)$ and the asymptotic properties of  $m(a)$, furthermore, we complete the proof of Theorem \ref{Thm1.5}.

\textbf{Notation.}
Throughout this paper we make use of the following notations.  \\
$\bullet$ The standard norm in $L^s(\mathbb{R}^N) (1\leq s\leq \infty)$ is denoted by $\|\cdot\|_s$.  \\
$\bullet$ $o_n(1)$ means a quantity which tends to 0 as $n$ tends to $\infty$.   \\
$\bullet$ The symbols $\rightarrow$ and $\rightharpoonup$ denote the strong and the weak convergence, respectively.  \\
$\bullet$ C, $C_\varepsilon, C_q$ stand for various constants whose exact values are irrelevant.

\section{Preliminaries}\label{preliminary}
In this section,  we assume $N=3,4$ and $p \in \left(4+\frac{4}{N}, 2\cdot2^*\right)$.
First of all, we take the perturbation method to deal with the nondifferentiability, which has been applied to constrained situation in \cite{JLW2015}.

For $\mu \in (0,1]$ and $\theta \in \left(\frac{4N}{N+2},\min\left\{N,\frac{4N+4}{N+2}\right\}\right)$, we consider the functional $I_\mu: \Upsilon \mapsto \mathbb{R} $
\begin{equation} \label{eq2-1}
I_{\mu}(u):=\frac{\mu}{\theta} \int_{\mathbb{R}^N} |\nabla u|^{\theta}+I(u),
\end{equation}
where $\Upsilon =W^{1,\theta}(\mathbb{R}^N)\cap H^1(\mathbb{R}^N)$. It is clear that $I_\mu \in \mathcal{C}^1(\Upsilon)$. At the same time, we define
\begin{equation} \label{eq2-2}
\mathcal{S}(a):=\left\{u \in \Upsilon \;\bigg|\;\int_{\mathbb{R}^N}|u|^2dx=a^2\right\}.
\end{equation}

It is easy to see that if $u$ is the critical point of $I_\mu |_{\mathcal{S}(a)}$, then $u$ satisfies the following Pohozaev identity
$$P_\mu (u):=\mu(1+\gamma_\theta) \int_{\mathbb{R}^N} |\nabla u|^{\theta} +\int_{\mathbb{R}^N} |\nabla u|^2 +
(N+2)\int_{\mathbb{R}^N} |u|^2 |\nabla u|^2 -\gamma_p \int_{\mathbb{R}^N} |u|^p ,$$
where $\gamma_\theta=\frac{N(\theta-2)}{2\theta}$.

By the Gagliardo-Nirenberg inequality \cite{A2008},
for any $u \in \{u \in L^2(\mathbb{R}^N)\;| \;\nabla (u^2) \in L^2(\mathbb{R}^N) \}$, there exists a constant $C_q >0$ such that
 \begin{equation}\label{GNI}
 \int_{\mathbb{R}^N} |u|^q \leq C_q\left (\int_{\mathbb{R}^N} |u|^2\right)^{\frac{4N-q(N-2)}{2(N+2)}}\left(4\int_{\mathbb{R}^N} |u|^2 |\nabla u|^2\right)^{\frac{N(q-2)}{2(N+2)}},\;\; \mbox{for}\;\; \mbox{any}\;\; q \in [2,2\cdot2^*].
 \end{equation}
Inspired by \cite{J1997}, we consider the $L^2$-norm preserved transform
\begin{equation} \label{eq2-5}
s*u(x)=e^{\frac{N}{2}s} u(e^sx),
\end{equation}
it is easy to see that
$$I_\mu(s*u)=\frac{\mu}{\theta} e^{\theta(1+\gamma_\theta)s}  \int_{\mathbb{R}^N} |\nabla u|^{\theta} + \frac{1}{2}e^{2s} \int_{\mathbb{R}^N} |\nabla u|^2 + e^{(N+2)s}\int_{\mathbb{R}^N} |u|^2 |\nabla u|^2 -\frac{1}{p}e^{p\gamma _ps}\int_{\mathbb{R}^N}|u|^p ,$$
\begin{align*}
\frac{d}{ds}I_\mu(s*u)=&\;\mu(1+\gamma_\theta) e^{\theta(1+\gamma_\theta)s}  \int_{\mathbb{R}^N} |\nabla u|^{\theta} + e^{2s} \int_{\mathbb{R}^N} |\nabla u|^2 + (N+2)e^{(N+2)s}\int_{\mathbb{R}^N} |u|^2 |\nabla u|^2  \\
&-\gamma_pe^{p\gamma _ps}\int_{\mathbb{R}^N}|u|^p.
\end{align*}
As a result,
$$P_\mu(u)=\frac{d}{ds}\bigg|_{s=0}I_\mu(s*u).$$

Next, we need to establish the following important lemmas.

\begin{lemma} \label{lem2.2}
There exists a constant $\zeta >0$ such that when $u \in \mathcal{S}(a)$ with $\int_{\mathbb{R}^N} |u|^2 |\nabla u|^2 <\zeta $, we have
\begin{equation}\label{eq2-10}
\frac{1}{4}\left(\int_{\mathbb{R}^N} |\nabla u|^2 +\int_{\mathbb{R}^N} |u|^2 |\nabla u|^2\right)\leq I_\mu(u) \leq   \frac{\mu}{\theta} \int_{\mathbb{R}^N} |\nabla u|^{\theta} +\int_{\mathbb{R}^N} |\nabla u|^2 +\int_{\mathbb{R}^N} |u|^2 |\nabla u|^2,
\end{equation}
\begin{equation}\label{eq2-10-1}
P_\mu (u)\geq \frac{1}{2} \left(\int_{\mathbb{R}^N} |\nabla u|^2 +\int_{\mathbb{R}^N} |u|^2 |\nabla u|^2\right).
\end{equation}
\end{lemma}

\begin{proof}
For any $u\in \mathcal{S}(a)$, taking into account \eqref{GNI}  we can deduce that
\begin{align*}
 \int_{\mathbb{R}^N} |u|^p &\leq C\left(\int_{\mathbb{R}^N} |u|^2 |\nabla u|^2\right)^{\frac{N(p-2)}{2(N+2)}}.
\end{align*}
It is clear that there exists a constant $\zeta' >\int_{\mathbb{R}^N} |u|^2 |\nabla u|^2$ such that
\begin{align*}
\left|\frac{1}{p}\int_{\mathbb{R}^N} |u|^p\right| \leq \frac{1}{4} \int_{\mathbb{R}^N} |u|^2 |\nabla u|^2 .
\end{align*}
Therefore,
\begin{align*}
\frac{1}{4}\left(\int_{\mathbb{R}^N} |\nabla u|^2 +\int_{\mathbb{R}^N} |u|^2 |\nabla u|^2\right)\leq I_\mu(u) \leq   \frac{\mu}{\theta} \int_{\mathbb{R}^N} |\nabla u|^{\theta} +\int_{\mathbb{R}^N} |\nabla u|^2 +\int_{\mathbb{R}^N} |u|^2 |\nabla u|^2.
\end{align*}

Similarly, there exists a constant $\zeta'' >\int_{\mathbb{R}^N} |u|^2 |\nabla u|^2$ such that
\begin{align*}
\left|\gamma_p\int_{\mathbb{R}^N} |u|^p\right| \leq \frac{1}{4} \int_{\mathbb{R}^N} |u|^2 |\nabla u|^2.
\end{align*}
Moreover,
\begin{align*}
P_\mu (u)&=\mu(1+\gamma_\theta) \int_{\mathbb{R}^N} |\nabla u|^{\theta} +\int_{\mathbb{R}^N} |\nabla u|^2 +(N+2)\int_{\mathbb{R}^N} |u|^2 |\nabla u|^2 -\gamma_p\int_{\mathbb{R}^N} |u|^p \\
&\geq \frac{1}{2} \left(\int_{\mathbb{R}^N} |\nabla u|^2 +\int_{\mathbb{R}^N} |u|^2 |\nabla u|^2\right).
\end{align*}

In the end, we may set $\zeta=\min\{\zeta',\zeta''\}$. Therefore, the conclusion is valid.

\end{proof}

Now, we introduce some properties about $I_\mu(s*u)$.
\begin{lemma} \label{lem2.3}
For any $u \in \Upsilon \backslash \{0\}$, we have

(1) $I_\mu(s*u) \rightarrow 0^+$ as $s\rightarrow -\infty;$

(2) $I_\mu(s*u) \rightarrow -\infty $ as $s\rightarrow +\infty.$
\end{lemma}

\begin{proof}
(1) There holds $\int_{\mathbb{R}^N} |s*u|^2 |\nabla (s*u)|^2 <\zeta $ for any $s$ satisfying $s<0$ and $|s|$ large enough. It follows from Lemma \ref{lem2.2} that
\begin{align*}
0\leq I_\mu(s*u) \leq \frac{\mu}{\theta} e^{\theta(1+\gamma_\theta)s}  \int_{\mathbb{R}^N} |\nabla u|^{\theta} + e^{2s} \int_{\mathbb{R}^N} |\nabla u|^2 + e^{(N+2)s}\int_{\mathbb{R}^N} |u|^2 |\nabla u|^2 \rightarrow 0^+,
\end{align*}
 as $s\rightarrow -\infty$.

(2) It is easy to see that
\begin{align*}
I_\mu(s*u) =&\ \frac{\mu}{\theta} e^{\theta(1+\gamma_\theta)s}  \int_{\mathbb{R}^N} |\nabla u|^{\theta} + \frac{1}{2}e^{2s} \int_{\mathbb{R}^N} |\nabla u|^2 + e^{(N+2)s}\int_{\mathbb{R}^N} |u|^2 |\nabla u|^2 -\frac{1}{p}e^{p\gamma _ps}\int_{\mathbb{R}^N}|u|^p \\
 \leq &\  e^{(N+2)s}\bigg(\frac{\mu}{\theta} e^{-[(N+2)-\theta(1+\gamma_\theta)]s}  \int_{\mathbb{R}^N} |\nabla u|^{\theta} + \frac{1}{2}e^{-Ns} \int_{\mathbb{R}^N} |\nabla u|^2 + \int_{\mathbb{R}^N} |u|^2 |\nabla u|^2   \\
 &-\frac{1}{p}e^{[p\gamma _p-(N+2)]s}\int_{\mathbb{R}^N}|u|^p\bigg) \\
 & \rightarrow -\infty,
\end{align*}
as $s\rightarrow +\infty$.
\end{proof}

We define the Pohozaev manifold $\mathcal{P}_\mu(a):=\{u\in \mathcal{S}(a)\ |\ P_\mu(u)=0 \}$ and obtain the following properties about the Pohozaev manifold.

\begin{lemma}\label{lem2.4}
For any $u \in \Upsilon \backslash \{0\}$, we have

(1) there exists a unique number $s(u) \in \mathbb{R}$ such that $P_\mu(s(u)*u)=0$, $I_\mu(s(u)*u)=\max_{s\in \mathbb{R}} I_\mu(s*u) >0$;

(2) the mapping $u\mapsto s(u)$ is well defined and continuous on $ \Upsilon \backslash \{0\}$;

(3) for any $y \in \mathbb{R}^N$, $s(u(\cdot+y)) = s(u)$ and $s(-u) = s(u)$ .

\end{lemma}

\begin{proof}
(1) It is easy to see that
$$ \frac{d}{ds}I_\mu(s*u)=P_\mu(s*u).$$
By using of Lemma \ref{lem2.3}, there exists some $s(u) \in \mathbb{R}$ such that $I_\mu(s*u)$ achieves the global maximum at $s(u)$ and $I_\mu(s(u)*u)>0$. Then, $ P_\mu(s(u)*u)=0.$

In what follows, we prove the uniqueness of $s(u)$. Suppose that $s_1(u)<s_2(u)$ such that
$$P_\mu(s_1(u)*u)=0=P_\mu(s_2(u)*u).$$
In view of the facts that
\begin{align*}
&\;e^{-(N+2)s}P_\mu(s*u)  \\
=&\;\mu(1+\gamma_\theta) e^{-[(N+2)-\theta(1+\gamma_\theta)]s}  \int_{\mathbb{R}^N} |\nabla u|^{\theta} + e^{-Ns} \int_{\mathbb{R}^N} |\nabla u|^2 + (N+2)\int_{\mathbb{R}^N} |u|^2 |\nabla u|^2 -\gamma_p e^{[p\gamma _p-(N+2)]s}\int_{\mathbb{R}^N}|u|^p   \\
=&:Q_\mu(s,u)
\end{align*}
and $Q_\mu(s,u)$ is strictly decreasing with respect to $s \in \mathbb{R}$,
we get $s_1(u)=s_2(u)$.

(2) By (1), it is easy to see that the mapping $u\mapsto s(u)$ is well defined on $ \Upsilon \backslash \{0\}.$
Let $\{u_n\} \subset \Upsilon \backslash \{0\}$ be any sequence satisfying $ u_n \rightarrow u \in \Upsilon \backslash \{0\}$ in $\Upsilon $.

\textbf{Claim:}
 $\{s(u_n)\}$ is bounded in $\mathbb{R}.$

Up to a subsequence, if $s(u_n) \rightarrow +\infty$ as $n \rightarrow \infty$, then
\begin{align*}
0\leq&\ e^{-(N+2)s(u_n)} I_\mu(s(u_n)*u_n) \\
\leq&\ \frac{\mu}{\theta}e^{-[(N+2)-\theta(1+\gamma_\theta)]s(u_n)}  \int_{\mathbb{R}^N} |\nabla u_n|^{\theta} + \frac{1}{2}e^{-Ns(u_n)} \int_{\mathbb{R}^N} |\nabla u_n|^2 + \int_{\mathbb{R}^N} |u_n|^2 |\nabla u_n|^2 \\
&-\frac{1}{p} e^{[p\gamma _p-(N+2)]s(u_n)}\int_{\mathbb{R}^N}|u|^p   \\
\rightarrow& -\infty,
\end{align*}
as $n \rightarrow \infty$, a contradiction, which implies that $\{s(u_n)\}$ is bounded from above.

Up to a subsequence, if $s(u_n) \rightarrow -\infty$ as $n \rightarrow \infty$, then it follows from $s(u)*u_n \rightarrow s(u)*u$ in $\Upsilon$ that
\begin{align*}
0&<I_\mu(s(u)*u) \\
&=\lim_{n\rightarrow \infty}I_\mu(s(u)*u_n) \\
&\leq \liminf_{n\rightarrow \infty}I_\mu(s(u_n)*u_n) \\
&\leq \liminf_{n\rightarrow \infty}\left( \frac{1}{\theta} \int_{\mathbb{R}^N} |\nabla (s(u_n)*u_n)|^{\theta} + \frac{1}{2}\int_{\mathbb{R}^N} |\nabla (s(u_n)*u_n)|^2 + \int_{\mathbb{R}^N} |s(u_n)*u_n|^2 |\nabla (s(u_n)*u_n)|^2 \right) \\
&\leq \liminf_{n\rightarrow \infty} \left(\frac{1}{\theta}e^{\theta(1+\gamma_\theta)s(u_n)}  \int_{\mathbb{R}^N} |\nabla u_n|^{\theta} + \frac{1}{2}e^{2s(u_n)} \int_{\mathbb{R}^N} |\nabla u_n|^2 + e^{(N+2)s(u_n)}\int_{\mathbb{R}^N} |u_n|^2 |\nabla u_n|^2 \right) \\
&=0,
\end{align*}
a contradiction. This shows $\{s(u_n)\}$ is bounded in $\mathbb{R}$.

Without loss of generality, we may assume that $s(u_n) \rightarrow s$ as $n \rightarrow \infty$, which together with the fact $u_n \rightarrow u$ in $\Upsilon$ yields that $s(u_n)*u_n \rightarrow s*u$ in $\Upsilon$. Hence, we get that
\begin{align*}
&\ P_\mu(s*u) \\
=&\ \mu(1+\gamma_\theta) \int_{\mathbb{R}^N} |\nabla (s*u)|^{\theta} + \int_{\mathbb{R}^N} |\nabla (s*u)|^2 + (N+2)\int_{\mathbb{R}^N} |s*u|^2 |\nabla (s*u)|^2 -\gamma_p \int_{\mathbb{R}^N} |s*u|^p  \\
=&\lim_{n\rightarrow \infty} \bigg[ \mu(1+\gamma_\theta) \int_{\mathbb{R}^N} |\nabla (s(u_n)*u_n)|^{\theta} + \int_{\mathbb{R}^N} |\nabla (s(u_n)*u_n)|^2 + (N+2)\int_{\mathbb{R}^N} |s(u_n)*u_n|^2 |\nabla (s(u_n)*u_n)|^2 \\
&-\gamma_p \int_{\mathbb{R}^N} |s(u_n)*u_n|^p \bigg] \\
= &\lim_{n\rightarrow \infty} P_\mu(s(u_n)*u_n) \\
= &\;0.
\end{align*}
Thus, we can conclude that $s(u)=s$ and $s(u_n) \rightarrow s(u)$ as $n \rightarrow \infty$.

(3) For any $y \in \mathbb{R}^N$, by changing variables in the integrals, we have
$$P_\mu(s(u)*u(\cdot + y)) = P_\mu(s(u)*u)=0,$$
thus $s(u(\cdot+ y)) = s(u)$ via (1). And it is clear that
$$P_\mu(s(u)*(-u)) = P_\mu(-s(u)*u)=P_\mu(s(u)*u)=0,$$
hence $s(-u)=s(u).$
\end{proof}

\begin{remark}\label{rem02.5}
    (1) For any $u \in \mathcal{S}(a)$, it follows from $s(u)*u \in \mathcal{P}_\mu(a)$ that $I_\mu(s(u)*u) \geq \inf_{u \in \mathcal{P}_\mu(a)}I_\mu(u)$, which yields
$\inf_{u \in \mathcal{S}(a)}\max_{s \in \mathbb{R}}I_\mu(s*u) \geq \inf_{u \in \mathcal{P}_\mu(a)}I_\mu(u)$. On the other hand, for any $u \in \mathcal{P}_\mu(a)$, one has $I_\mu(u) = I_\mu(0*u) \geq \inf_{u \in \mathcal{S}(a)}I_\mu(s(u)*u)$, i.e., $\inf_{u \in \mathcal{S}(a)}\max_{s \in \mathbb{R}}I_\mu(s*u) \leq \inf_{u \in \mathcal{P}_\mu(a)}I_\mu(u)$.
So we can conclude
$$\inf_{u \in \mathcal{P}_\mu(a)}I_\mu(u)=\inf_{u \in \mathcal{S}(a)}\max_{s \in \mathbb{R}}I_\mu(s*u).$$
(2) For $P(u)$, it is easy to see (1) and (3) are valid; since the quasilinear term is discontinuous in $H^1(\mathbb{R}^N)$, we can not obtain (2), but we can get $s(u_n) \rightarrow s(u)$ for $\int_{\mathbb{R}^N} |u_n|^2 |\nabla u_n|^2 \rightarrow \int_{\mathbb{R}^N} |u|^2 |\nabla u|^2$ and $u_n \rightarrow u \neq 0$ in $H^1(\mathbb{R}^N)$.   \\
(3) It is clear that we can write $s(u)$ as $s_\mu(u)$. For any $0<\mu_1<\mu_2<1$, we have
$$Q_{\mu_1}(s_{\mu_1}(u),u)=0=Q_{\mu_2}(s_{\mu_2}(u),u)>Q_{\mu_1}(s_{\mu_2}(u),u),$$
which implies $s_{\mu_2}(u)>s_{\mu_1}(u)$. So, there exist a $C(u)$ such that $s_\mu(u) \leq C(u)$ for any fixed $u \in \Upsilon \backslash \{0\}$ and $0<\mu<1$.
\end{remark}

\begin{lemma}\label{lem2.5}
For $\mu \in (0,1]$, we have

(1) $m_\mu(a):=\inf_{u \in \mathcal{P}_\mu(a)} I_\mu(u) \geq \frac{\zeta}{8}>0$;

(2) let $\{u_n\} \subset \mathcal{P}_\mu(a)$, if  $\sup_{n \in \mathbb{N}^+} I_\mu(u_n)<+\infty$, then there exists a constant $C>0$ such that
$$ \mu \int_{\mathbb{R}^N} |\nabla u_n|^{\theta}\leq C,\; \int_{\mathbb{R}^N} |\nabla u_n|^2 \leq C,\;  \int_{\mathbb{R}^N} |u_n|^2 |\nabla u_n|^2 \leq C, \;for\;any \;n \in \mathbb{N}^+.$$
\end{lemma}

\begin{proof}
(1) For $u \in \mathcal{P}_\mu(a)$, we take $s$ such that
$$ e^{-(N+2)s}\int_{\mathbb{R}^N} |u|^2 |\nabla u|^2=\frac{\zeta}{2}.$$
Set $\tilde{u}:=(-s)*u$, it is clear that
$$\int_{\mathbb{R}^N} |\tilde{u}|^2 |\nabla \tilde{u}|^2=\frac{\zeta}{2}<\zeta,$$
 by using of Lemma \ref{lem2.2}, we can deduce that
\begin{align*}
&\; I_\mu(u) =I_\mu(0*u) \geq I_\mu(-s*u)  \\
\geq&\; \frac{1}{4} \left(\int_{\mathbb{R}^N} |\nabla \tilde{u}|^2 +\int_{\mathbb{R}^N} |\tilde{u}|^2 |\nabla \tilde{u}|^2\right) \\
\geq&\; \frac{\zeta}{8} >0.
\end{align*}

(2) From $u_n \in \mathcal{P}_\mu(a)$, we can deduce that
\begin{align*}
&\ \mu(1+\gamma_\theta) \int_{\mathbb{R}^N} |\nabla u_n|^{\theta} +\int_{\mathbb{R}^N} |\nabla u_n|^2 +(N+2)\int_{\mathbb{R}^N} |u_n|^2 |\nabla u_n|^2 \\
=& \ \gamma_p \int_{\mathbb{R}^N} |u_n|^p
>\ \frac{(N+2)}{p}\int_{\mathbb{R}^N} |u_n|^p.
\end{align*}
Obviously,
\begin{align*}
I_\mu(u_n)&=\frac{\mu}{\theta} \int_{\mathbb{R}^N} |\nabla u_n|^{\theta} +\frac{1}{2}\int_{\mathbb{R}^N} |\nabla u_n|^2 +\int_{\mathbb{R}^N} |u_n|^2 |\nabla u_n|^2 -\frac{1}{p} \int_{\mathbb{R}^N} |u_n|^p \\
&>\frac{(N+2)-\theta(1+\gamma_\theta)}{(N+2)\theta} \mu \int_{\mathbb{R}^N} |\nabla u_n|^{\theta} +\frac{N}{2(N+2)}\int_{\mathbb{R}^N} |\nabla u_n|^2,
\end{align*}
which implies
$$ \mu \int_{\mathbb{R}^N} |\nabla u_n|^{\theta}\leq C\;\;\mbox{and} \; \int_{\mathbb{R}^N} |\nabla u_n|^2 \leq C.$$
Next, we prove that
$$\limsup_{n \rightarrow \infty}\int_{\mathbb{R}^N} |u_n|^2 |\nabla u_n|^2 \leq C.$$
Assume that up to a subsequence, $\int_{\mathbb{R}^N} |u_n|^2 |\nabla u_n|^2 \rightarrow +\infty$ as $n \rightarrow \infty$. Take $s_n$ such that
$$e^{-\theta(1+\gamma_\theta)s_n} \left(\mu \int_{\mathbb{R}^N} |\nabla u_n|^{\theta}\right) +e^{-2s_n} \int_{\mathbb{R}^N} |\nabla u_n|^2 + e^{-(N+2)s_n}\int_{\mathbb{R}^N} |u_n|^2 |\nabla u_n|^2=1.$$
Set $v_n:=(-s_n)*u_n$, it is clear that
$$ \mu \int_{\mathbb{R}^N} |\nabla v_n|^{\theta} + \int_{\mathbb{R}^N} |\nabla v_n|^2 +\int_{\mathbb{R}^N} |v_n|^2 |\nabla v_n|^2=1,\;  \int_{\mathbb{R}^N} |v_n|^2= \int_{\mathbb{R}^N} |u_n|^2=a^2\;\; \mbox{and}\;\; s(v_n)=s_n \rightarrow +\infty.$$
Let
$$\rho :=\limsup_{n \rightarrow \infty} \sup_{y \in \mathbb{R}^N} \int_{B_1(y)} |v_n|^2.$$

\textbf{Case 1:}
 If $\rho=0$, then $v_n \rightarrow 0$ in $L^{2+\frac{4}{N}}(\mathbb{R}^N)$ by \cite[Lemma 1.21]{Willem}. It follows from the Interpolation inequality that for any $n>0$ large enough, we have
\begin{align*}
    \int_{\mathbb{R}^N} |v_n|^p &\leq \left(\int_{\mathbb{R}^N} |v_n|^{2\cdot2^*} \right)^\frac{[(p-2)N-4](N-2)}{2(N^2+4)}\left(\int_{\mathbb{R}^N} |v_n|^{2+\frac{4}{N}} \right)^\frac{[4N-p(N-2)]N}{2(N^2+4)}  \\
    &\leq C\left(\int_{\mathbb{R}^N} |v_n|^2 |\nabla v_n|^2\right)^\frac{[(p-2)N-4]N}{2(N^2+4)}\left(\int_{\mathbb{R}^N} |v_n|^{2+\frac{4}{N}} \right)^\frac{[4N-p(N-2)]N}{2(N^2+4)}  \\
    &\leq \varepsilon,
\end{align*}
which means
\begin{equation}\label{eq02-12}
    \lim_{n \rightarrow \infty} \frac{1}{p}e^{p\gamma_ps}\int_{\mathbb{R}^N} |v_n|^p=0 \;\; \mbox{for} \;\; \mbox{any}\;\; s>0.
\end{equation}
Since $P_\mu(s(v_n)*v_n)=P_\mu(u_n)=0$, we obtain that for $s>0$,
\begin{align*}
    C&\geq I_\mu(s(v_n)*v_n) \geq I_\mu(s*v_n) \\
      &= \frac{1}{\theta} e^{\theta(1+\gamma_\theta)s} \left(\mu \int_{\mathbb{R}^N} |\nabla v_n|^{\theta}\right) + \frac{1}{2}e^{2s} \int_{\mathbb{R}^N} |\nabla v_n|^2 + e^{(N+2)s}\int_{\mathbb{R}^N} |v_n|^2 |\nabla v_n|^2 -\frac{1}{p}e^{p\gamma_ps}\int_{\mathbb{R}^N} |v_n|^p \\
      &\geq \frac{1}{N}e^{2s} \left(\mu \int_{\mathbb{R}^N} |\nabla v_n|^{\theta} + \int_{\mathbb{R}^N} |\nabla v_n|^2 +\int_{\mathbb{R}^N} |v_n|^2 |\nabla v_n|^2\right) +o_n(1) \\
      &=\frac{1}{N}e^{2s} +o_n(1).
\end{align*}
Clearly, this leads to contradiction for $s>\max\left\{0,\frac{\ln(NC)}{2}\right\}$.

\textbf{Case 2:}
    If $\rho >0$, up to a subsequence, we may assume the existence of $\{y_n\} \subset \mathbb{R}^N$ such that
    $$\int_{B_1(y_n)} |v_n|^2> \frac{\rho}{2},$$
by changing variables in the integrals, we have
 $$\int_{B_1(0)} |v_n(\cdot +y_n)|^2> \frac{\rho}{2}.$$
Since $\{v_n(\cdot +y_n)\}$ is bounded in $H^1(\mathbb{R}^N)$, we suppose, up to a subsequence
\begin{align*}
 &v_n(\cdot +y_n) \rightharpoonup w     \;\;\;\,\;\;\; \;\;\;\;\;\;\;\;\;\;\;  \mbox{in} \; H^1(\mathbb{R}^N) , \\
 &v_n(\cdot +y_n) \rightarrow  w \neq 0       \;\;\;\;\;\;\;\;\;\;\;   \mbox{in} \;L^2_{loc}(\mathbb{R}^N),\\
 &v_n(\cdot +y_n) \rightarrow  w   \;\;\;\;\;\;\;\;\;\;\;\;\; \;\;\;\; \,             \mbox{a.e.} \; \mbox{in} \; \mathbb{R}^N.
\end{align*}
Set $z_n:=v_n(\cdot +y_n)$. Thus, Lemma \ref{lem2.4} (3) gives us that
$$s(z_n)=s(v_n(\cdot +y_n))=s(v_n) \rightarrow + \infty.$$
As a result,
\begin{align*}
0\leq&\ e^{-(N+2)s(z_n)} I_\mu(s(z_n)*z_n) \\
=&\ \frac{1}{\theta}e^{-[(N+2)-\theta(1+\gamma_\theta)]s(z_n)}  \left(\mu \int_{\mathbb{R}^N} |\nabla z_n|^{\theta}\right) + \frac{1}{2}e^{-Ns(z_n)} \int_{\mathbb{R}^N} |\nabla z_n|^2 + \int_{\mathbb{R}^N} |z_n|^2 |\nabla z_n|^2 \\
&-\frac{1}{p}e^{[p\gamma_p-(N+2)]s(z_n)} \int_{\mathbb{R}^N} |z_n|^p \\
\rightarrow &-\infty,
\end{align*}
as $n \rightarrow \infty$, a contradiction. Therefore,
$$\limsup_{n \rightarrow \infty}\int_{\mathbb{R}^N} |u_n|^2 |\nabla u_n|^2 \leq C.$$
So the conclusion holds.

\end{proof}

\begin{remark}\label{rem3.5}
It is clear that $m(a)\geq \frac{\zeta}{8}>0$.
\end{remark}

\begin{lemma}\label{lambda}
For $N=3,4$, $p \in \left(4+\frac{4}{N},2\cdot2^*\right]$, $u \in \mathcal{S}'(a)$ and there exists a $\lambda \in \mathbb{R}$ such that $\Theta(u,\lambda)=0$. Then $\lambda >0$.
\end{lemma}

\begin{proof}
It is easy to see that $\Theta(|u|,\lambda)=0$. Assume $\lambda \leq 0$, we make use of a change of variables which was first introduced in \cite{CJ2004,LWW2003}
\begin{align*}
v=f(|u|)=\int^{|u|}_0\sqrt{1+2t^2}dt \geq 0,
\end{align*}
It is clear that the function $v = f(|u|)$ is strictly monotone and hence has an inverse function, denoted by $|u| = g(v)$. Using the transformation $|u| = g(v)$, we have
\begin{align*}
    -\Delta v=\frac{1}{\sqrt{1+2|g(v)|^2}}\left(|g(v)|^{p-2}g(v)-\lambda g(v)\right)\geq 0.
\end{align*}
By \cite[Lemma2.6]{LLW2013-2}, we get $|u| \in C^{2,\alpha}$ for some $\alpha >0$, as a result, $v \in C^{2}$. At the same time, it is easy to see that
\begin{align*}
\begin{split}
v\leq
\left \{
      \begin{array}{ll}
           C|u|,  &|u| \leq 1, \\
           C|u|^2,  &|u| >1. \\
      \end{array}
\right.
\end{split}
\end{align*}
Therefore, $v \in L^\frac{N}{N-2}(\mathbb{R}^N)$. It follows from the Liouville type result \cite[Lemma A.2]{NI2014} that $u\equiv 0$. This is impossible. Thus, $\lambda >0$.
\end{proof}

In general, Palais-Smale type compactness condition is needed. However, it seems impossible to prove that $I_\mu$ has such a compactness property. To overcome this difficulty, We consider an auxiliary functional as the one in \cite{JL2020}:
\begin{align*}
\Psi_\mu(u):&=I_\mu(s(u)*u)   \\
&=\frac{\mu}{\theta} e^{\theta(1+\gamma_\theta)s(u)}  \int_{\mathbb{R}^N} |\nabla u|^{\theta} + \frac{1}{2}e^{2s(u)} \int_{\mathbb{R}^N} |\nabla u|^2 + e^{(N+2)s(u)}\int_{\mathbb{R}^N} |u|^2 |\nabla u|^2 -\frac{1}{p}e^{p\gamma_ps(u)}\int_{\mathbb{R}^N}|u|^p ,
\end{align*}
where $s(u)$ is given by Lemma \ref{lem2.4}.

\begin{lemma}\label{lem5.2}
For $\mu \in (0,1]$, $u \in \Upsilon \backslash \{0\}$ and $\varphi \in \Upsilon$, we have
\begin{align*}
\Psi'_\mu(u)[\varphi]=&\ \mu e^{\theta(1+\gamma_\theta)s(u)} \int_{\mathbb{R}^N} |\nabla u|^{\theta-2}\nabla u \cdot \nabla \varphi       +e^{2s(u)}\int_{\mathbb{R}^N} \nabla u \cdot \nabla \varphi  \\
&+2e^{(N+2)s(u)}\int_{\mathbb{R}^N} |u|^2 \nabla u \cdot \nabla\varphi + u\varphi |\nabla u|^2 -e^{p\gamma_p s(u)}\int_{\mathbb{R}^N} |u|^{p-2}u\varphi  \\
 =&\ I'_\mu(s(u)*u)[s(u)*\varphi].
\end{align*}
\end{lemma}

\begin{proof}
Let $u \in \Upsilon \backslash \{0\}$, $\varphi \in \Upsilon$, $u_t=u+t\varphi$ and $s_t=s(u_t)$ with $|t|$ small enough. It follows from the mean value theorem that
\begin{align*}
\Psi_\mu(u_t)-\Psi_\mu(u) =&\ I_\mu(s_t*u_t)-I_\mu(s_0*u) \leq I_\mu(s_t*u_t)-I_\mu(s_t*u)  \\
  =&\ \mu e^{\theta(1+\gamma_\theta)s_t} \int_{\mathbb{R}^N} \left(|\nabla u_{\xi_t}|^{\theta-2}\nabla u_{\xi_t} \cdot \nabla \varphi\right)  t +e^{2s_t}\int_{\mathbb{R}^N} \left(\nabla u_{\xi_t} \cdot \nabla \varphi\right) t  \\
&+2e^{(N+2)s_t}\int_{\mathbb{R}^N} \left(|u_{\xi_t}|^2 \nabla u_{\xi_t} \cdot \nabla \varphi
 + u_{\xi_t}\varphi |\nabla u_{\xi_t}|^2\right)t -e^{p\gamma_p s_t}\int_{\mathbb{R}^N} \left(|u_{\xi_t}|^{p-2}u_{\xi_t}\varphi\right) t ,
\end{align*}
where $|\xi_t| \in (0,|t|).$
On the other hand,
\begin{align*}
\Psi_\mu(u_t)-\Psi_\mu(u) =&\ I_\mu(s_t*u_t)-I_\mu(s_0*u) \geq I_\mu(s_0*u_t)-I_\mu(s_0*u)  \\
  =&\ \mu e^{\theta(1+\gamma_\theta)s_0} \int_{\mathbb{R}^N} \left(|\nabla u_{\kappa_t}|^{\theta-2}\nabla u_{\kappa_t} \cdot \nabla \varphi\right) t +e^{2s_0}\int_{\mathbb{R}^N} \left(\nabla u_{\kappa_t} \cdot \nabla\varphi\right) t  \\
&+2e^{(N+2)s_0}\int_{\mathbb{R}^N} \left(|u_{\kappa_t}|^2 \nabla u_{\kappa_t} \cdot \nabla \varphi
 + u_{\kappa_t}\varphi |\nabla u_{\kappa_t}|^2\right)t -e^{p\gamma_p s_0}\int_{\mathbb{R}^N} \left(|u_{\kappa_t}|^{p-2}u_{\kappa_t}\varphi\right) t  ,
\end{align*}
where $|\kappa_t| \in (0,|t|).$

Since $\lim_{t \rightarrow 0} s_t = s_0$, we obtain that
\begin{align*}
\Psi'_\mu(u)[\varphi]=&\ \lim_{t \rightarrow 0} \frac{\Psi_\mu(u_t)-\Psi_\mu(u)}{t}  \\
=&\ \mu e^{\theta(1+\gamma_\theta)s(u)} \int_{\mathbb{R}^N} |\nabla u|^{\theta-2}\nabla u \cdot \nabla \varphi       +e^{2s(u)}\int_{\mathbb{R}^N} \nabla u \cdot \nabla \varphi  \\
&+2e^{(N+2)s(u)}\int_{\mathbb{R}^N} |u|^2 \nabla u \cdot \nabla\varphi + u\varphi |\nabla u|^2 -e^{p\gamma_p s(u)}\int_{\mathbb{R}^N} |u|^{p-2}u\varphi .
\end{align*}
By changing variables in the integrals, we can get that
$$\Psi'_\mu(u)[\varphi]=I'_\mu(s(u)*u)[s(u)*\varphi].$$
\end{proof}

\section{The Sobolev critical case}\label{S-case}

In this section, we will finish the proof of Theorem \ref{Thm1}.

\begin{proof} [Proof of Theorem \ref{Thm1}]
If $N=3,4$ and $p=2\cdot2^*$, we assume that there exists a $\lambda \in \mathbb{R}$ such that $\Theta(u,\lambda)=0$. Then, we conclude $P(u)=0$, which implies
\begin{align*}
    (1-\gamma_{2\cdot2^*})\int_{\mathbb{R}^N} |\nabla u|^2=\lambda \gamma_{2\cdot2^*} \int_{\mathbb{R}^N} |u|^{2}.
\end{align*}
Therefore, $\lambda <0$. At the same time, by Lemma \ref{lambda}, we have $\lambda >0$. This is impossible.
\end{proof}

\section{The case of non-radial space }\label{noradial space}
In this section, we assume $N=3$ and $p \in \left(4+\frac{4}{N}, 2^*\right]$. We recall the following definition and minimax principe under the standard boundary
condition. By these, a technical result can be established, which shows the existence
of a “nice" Palais-Smale sequence with an additional property.

\begin{definition}\label{def3.1}
\cite[Definition 3.1]{G1993} Let B be a closed subset of X. We say that a class $\mathcal{F}$ of compact subsets of X is a homotopy stable family with boundary B provided

(a) every set in $\mathcal{F}$ contains B;

(b) for any set A in $\mathcal{F}$ and any $\eta \in \mathcal{C}([0,1]\times X,X)$ satisfying $\eta(t,x)=x$ for all $(t,x)$ in $(\{0\}\times X)\cup ([0,1]\times B) $ we have that $\eta(1,A) \in\mathcal{F}.$
\end{definition}
The above definition is still valid if the boundary $B =\emptyset$.

\begin{proposition}\label{prop3.2}
\cite[Theorem 3.2]{G1993} Let $\phi$ be a $\mathcal{C}^1$-functional on a complete connected $\mathcal{C}^1$-Finsler manifold X (without boundary) and consider a homotopy stable family $\mathcal{F}$ of compact subsets of X with a closed boundary B. Set $c=c(\phi,\mathcal{F})$ and suppose that
\begin{equation}\label{eq3.5}
\sup\phi(B) < c.
\end{equation}
Then for any sequence of sets $\{A_n\} \subset \mathcal{F} $ satisfying $\lim_{n\rightarrow \infty} \sup_{A_n} \phi=c$, there exists a sequence $x_n \subset X$ such that

(1) $\lim_{n\rightarrow\infty}\phi(x_n)=c;$

(2) $\lim_{n\rightarrow\infty}\|\phi'(x_n)\|=0;$

(3) $\lim_{n\rightarrow\infty}{\rm dist}(x_n,A_n)=0.$  \\
Moreover, if $\phi'$ is uniformly continuous, then $x_n$ can be chosen to be in $A_n$ for each
n.
\end{proposition}

\begin{lemma}\label{lem-P}
Assume $\mu \in(0,1]$, then $\mathcal{P}_\mu(a)$ is a $\mathcal{C}^1$-submanifold of codimension 1 in $\mathcal{S}(a)$.
\end{lemma}

\begin{proof}
We define $L(a):=\int_{\mathbb{R}^N}|u|^2-a^2$, clearly $L \in \mathcal{C}^1 (\Upsilon)$.

\textbf{Claim:}
 $d(P_\mu,L): \Upsilon \mapsto \mathbb{R}^2$ is surjective.

If $dP_\mu(u)$ and $dL(u)$ are linearly dependent, then there exists a $l \in \mathbb{R}$ such that
\begin{align*}
&\,2l \int_{\mathbb{R}^N}u\varphi  \\
=&\;\theta \mu (1+\gamma_\theta) \int_{\mathbb{R}^N} |\nabla u|^{\theta-2}\nabla u \cdot \nabla \varphi
+2\int_{\mathbb{R}^N} \nabla u \cdot \nabla \varphi
+2(N+2)\int_{\mathbb{R}^N}\left(|u|^2 \nabla u \cdot \nabla \varphi +u\varphi|\nabla u|^2\right)    \\
&-p\gamma_p \int_{\mathbb{R}^N}|u|^{p-2}u\varphi
\end{align*}
for any $\varphi \in \Upsilon$. Similar to Pohozaev identity, let $\varphi=u$ and $\varphi=x\cdot \nabla u$, we obtain
\begin{align*}
&\theta\mu (1+\gamma_\theta)^2 \int_{\mathbb{R}^N} |\nabla u|^{\theta}  +2\int_{\mathbb{R}^N} |\nabla u|^2
+(N+2)^2\int_{\mathbb{R}^N}|u|^2|\nabla u|^2
-p\gamma_p^2 \int_{\mathbb{R}^N}|u|^p=0.
\end{align*}
Due to $P_\mu (u)=0 $, we conclude
\begin{align*}
&[(N+2)-\theta (1+\gamma_\theta)](1+\gamma_\theta)\mu \int_{\mathbb{R}^N} |\nabla u|^{\theta}  +N\int_{\mathbb{R}^N} |\nabla u|^2
+\gamma_p[p\gamma_p-(N+2)]\int_{\mathbb{R}^N}|u|^p=0.
\end{align*}
Thus $u=0$. This contradicts with $u \in \mathcal{S}(a).$

\end{proof}

Next, we will prove the existence of Palais-Smale sequence.

\begin{lemma} \label{lem3.4.1}
For any fixed $\mu \in (0,1]$, assume that $\mathcal{F}$ is a homotopy stable family of compact subsets of $\mathcal{S}(a)$ with
$B = \emptyset$, define the level
$$E_\mu^\mathcal{F}(a):=\inf_{A \in \mathcal{F}}\max_{u \in A} \Psi_\mu(u).$$  \\
If $E_\mu^\mathcal{F}(a)>0$, then there exists a Palais-Smale sequence $\{u_n\} \subset \mathcal{P}_\mu(a)$ for $I_\mu$ constrained on $\mathcal{S}(a)$ at the
level $E_\mu^\mathcal{F}(a)$.
\end{lemma}

\begin{proof}
Taking $\{A_n\} \subset \mathcal{F}$ be a minimizing sequence for $E_\mu^\mathcal{F}(a)$. Define $\eta:[0,1]\times \mathcal{S}(a) \rightarrow \mathcal{S}(a)$ by
$$\eta(t,u)=(ts(u))*u,$$
by lemma \ref{lem2.4} (2), one can easily see that $\eta$ is continuous. Moreover, it follows that $\eta(t,u)=u$ for $(t,u) \in \{0\} \times \mathcal{S}(a)$. By the definition of $\mathcal{F}$, we have
$$D_n:=\eta(1,A_n)=\{s(u)*u \,|\, u \in A_n \} \in \mathcal{F}.$$
Clearly, $D_n \subset \mathcal{P}_\mu(a)$ for $n \in \mathbb{N^+}$. Since for any $u \in A_n$
\begin{align*}
    \Psi_\mu(s(u)*u)=I_\mu(0*(s(u)*u))=I_\mu(s(u)*u)=\Psi_\mu(u),
\end{align*}
we infer that
$$\max_{D_n}\Psi_\mu=\max_{A_n}\Psi_\mu \rightarrow E_\mu^\mathcal{F}(a)$$
and hence $\{D_n\} \subset \mathcal{F}$ is also a minimizing sequence
of $E_\mu^\mathcal{F}(a)$. Therefore Proposition \ref{prop3.2} yields a Palais-Smale sequence ${v_n} \subset \Upsilon$ for $\Psi_\mu$
on $\mathcal{S}(a)$ at the level $E_\mu^\mathcal{F}(a)$, that is, as $n \rightarrow \infty$, $\{v_n\}$ satisfies the following property
\begin{align*}
    \Psi_\mu(v_n) \rightarrow E_\mu^\mathcal{F}(a),\ dist(v_n,D_n) \rightarrow 0,\ \|\Psi_\mu^{'} (v_n)\|_{v_n,*} \rightarrow 0.
\end{align*}
Hereafter, set
\begin{align*}
    s_n:=s(v_n) \ \mbox{and} \  u_n:=s_n*v_n=s(v_n)*v_n \in \mathcal{P}_\mu(a).
\end{align*}

\textbf{Claim:} There exists some $C>0$ such that $-s_n \leq C$ for any $n \in \mathbb{N}^+.$

Indeed, a direct computation shows that
$$\int_{\mathbb{R}^N} |u_n|^2 |\nabla u_n|^2=e^{(N+2)s_n}\int_{\mathbb{R}^N} |v_n|^2 |\nabla v_n|^2.$$
The fact ${u_n} \subset \mathcal{P}_\mu(a)$, combining with the inequality \eqref{GNI}, implies that $\left\{\int_{\mathbb{R}^N} |u_n|^2 |\nabla u_n|^2\right\}$ is bounded
from below by a positive constant. To complete the proof of Claim, it remains to show
that $\sup_n \int_{\mathbb{R}^N} |v_n|^2 |\nabla v_n|^2< \infty$. Since $D_n \subset \mathcal{P}_\mu(a)$, one can easily see
$$\max_{D_n}I_\mu=\max_{D_n}\Psi_\mu \rightarrow E_\mu^\mathcal{F}(a)>0$$
and it follows from Lemma \ref{lem2.5} (2) that ${D_n}$ is uniformly bounded in $\Upsilon$. Then,
due to $dist(v_n, D_n) \rightarrow 0$, we have $\sup_n
\int_{\mathbb{R}^N} |v_n|^2 |\nabla v_n|^2< \infty$. And it is easy to see that $-s_n \leq C$ for any $n \in \mathbb{N}^+.$

Using the fact $\{u_n\} \subset \mathcal{P}_\mu(a)$ again, one has
$$I_\mu(u_n)=I_\mu(s_n*v_n)=\Psi_\mu(v_n) \rightarrow E_\mu^\mathcal{F}(a)$$
and
\begin{align*}
\|I'_\mu|_{\mathcal{S}(a)}(u_n) \|_{u_n,*} &=\sup_{\psi \in T_{u_n}\mathcal{S}(a),\|\psi\|_\Upsilon\leq1} |I'_\mu(u_n)[\psi]| \\
&=\sup_{\psi \in T_{u_n}\mathcal{S}(a),\|\psi\|_\Upsilon\leq1} |I'_\mu(s_n*v_n)[s_n*(-s_n)*\psi]|  \\
&=\sup_{\psi \in T_{u_n}\mathcal{S}(a),\|\psi\|_\Upsilon\leq1} | \Psi'_\mu(v_n)[(-s_n)*\psi]|  \\
&\leq  \| \Psi'_\mu(v_n)\|_{v_n,*} \sup_{\psi \in T_{u_n}\mathcal{S}(a),\|\psi\|_\Upsilon\leq1} \|(-s_n)*\psi\|  \\
&\leq C\|\Psi'_\mu(v_n)\|_{v_n,*} \rightarrow 0.
\end{align*}
The proof is complete.
\end{proof}

In order to obtain compactness, we need to establish the following important lemmas.

\begin{lemma}\label{lem6.1'}
    Assume $0<\mu \leq 1$ and $4+\frac{4}{N}<p\leq 2^*$, $m_\mu(a)$ is nonincreasing with respect to $a$.
\end{lemma}
\begin{proof}
    In fact, for $0<b\leq a$ and $u \in \mathcal{S}(b)$ satisfying $\max_{s_0 \in \mathbb{R}}I_\mu(s_0*u) \leq m_\mu(b) + \varepsilon$, set $k=\frac{a}{b} \geq 1$, then $\tilde{u}:=ku \in \mathcal{S}(a)$ and
    \begin{align*}
        &\,m_\mu(a)   \\
        \leq& \max_{s \in \mathbb{R}}I_\mu(s*\tilde{u})  \\
        \leq& \max_{s \in \mathbb{R}} \left\{ \frac{\mu}{\theta} e^{\theta(1+\gamma_\theta)s} k^\theta \int_{\mathbb{R}^N} |\nabla u|^{\theta} + \frac{1}{2}e^{2s} k^2\int_{\mathbb{R}^N} |\nabla u|^2 + e^{(N+2)s}k^4\int_{\mathbb{R}^N} |u|^2 |\nabla u|^2 -\frac{1}{p}e^{p\gamma _ps}k^p\int_{\mathbb{R}^N}|u|^p   \right\} \\
        \leq& \max_{s \in \mathbb{R}} \left\{  \frac{\mu}{\theta} \left( ke^s\right)^{\theta(1+\gamma_\theta)}  \int_{\mathbb{R}^N} |\nabla u|^{\theta} + \frac{1}{2}\left( ke^s\right)^{2} \int_{\mathbb{R}^N} |\nabla u|^2 + \left( ke^s\right)^{(N+2)}\int_{\mathbb{R}^N} |u|^2 |\nabla u|^2 -\frac{1}{p}\left( ke^s\right)^{p\gamma _p}\int_{\mathbb{R}^N}|u|^p    \right\}  \\
        \leq& \max_{s_0 \in \mathbb{R}} \left\{\frac{\mu}{\theta} e^{\theta(1+\gamma_\theta)s_0}  \int_{\mathbb{R}^N} |\nabla u|^{\theta} + \frac{1}{2}e^{2s_0} \int_{\mathbb{R}^N} |\nabla u|^2 + e^{(N+2)s_0}\int_{\mathbb{R}^N} |u|^2 |\nabla u|^2 -\frac{1}{p}e^{p\gamma _ps_0}\int_{\mathbb{R}^N}|u|^p  \right\}  \\
        \leq&\, m_\mu(b) + \varepsilon,
    \end{align*}
    thus $m_\mu(a)\leq m_\mu(b)$.
\end{proof}

\begin{remark}
    It is easy to see that $m(a)$ is nonincreasing with respect to $a$ for $4+\frac{4}{N}<p\leq 2^*$.
\end{remark}

\begin{lemma}
    If $a>0$ and $\mu >0$, then for any $\varepsilon >0$, there exists $\delta \in (0,1)$ such that $m_\mu(a) \leq m(a)+\varepsilon$ for $\mu \in (0,\delta)$.
\end{lemma}
\begin{proof}
Take $w$ be the Schwarz symmetric function such that $I(w)= m(a)$ and $\Theta(w,\lambda)=0$ by \cite{ZZJ2025} or \cite{ZCW2023}. And $w$ is nonnegative, radially symmetric and nonincreasing. By \cite[Lemma2.6]{LLW2013-2}, we get $w \in C^{2,\alpha}$ for some $\alpha >0$. It is easy to see that $w \in \mathcal{S}'(a) \cap C^2(\mathbb{R}^N) \cap L^\infty(\mathbb{R}^N)$. Set $\phi \in C^\infty(\mathbb{R}^N)$
be a radial function such that $\phi = 1$ for $|x| \leq R$ and $\phi=0$
for $|x| \geq 2R$, then $w_R := \frac{a \phi w}{\|\phi w\|_2} \in \mathcal{S}(a)$.
It follows from Lemma \ref{lem2.4} that
\begin{align*}
    m_\mu(a) \leq \max_{s \in \mathbb{R}} I_\mu(s*w_R)
    = &\; \frac{\mu}{\theta} \int_{\mathbb{R}^N} |\nabla(s_\mu(w_R)*w_R)|^\theta +I(s_\mu(w_R)*w_R)   \\
    \leq &\; \frac{\mu}{\theta} \int_{\mathbb{R}^N} |\nabla(s_\mu(w_R)*w_R)|^\theta +I(s_0(w_R)*w_R),
\end{align*}
where $I(s_0(w_R)*w_R)=\max_{s \in \mathbb{R}} I(s*w_R)$.
By Remark \ref{rem02.5}(3), we conclude that
\begin{align*}
    \frac{\mu}{\theta} \int_{\mathbb{R}^N} |\nabla(s_\mu(w_R)*w_R)|^\theta = \frac{\mu}{\theta} e^{\theta(1+\gamma_\theta)s_\mu(w_R)}\int_{\mathbb{R}^N} |\nabla(w_R)|^\theta \leq \mu C(w_R)= \mu C(R).
\end{align*}
It is clear that $\int_{\mathbb{R}^N} |w_R|^2 |\nabla w_R|^2 \rightarrow \int_{\mathbb{R}^N} |w|^2 |\nabla w|^2$ and $w_R \rightarrow w$ in $H^1(\mathbb{R}^N)$.
And it follows from Remark \ref{rem02.5}(2) that $s_0(w_R) \rightarrow s_0(w)=0$, which implies
\begin{align*}
   m_\mu(a) &\leq I(w)+o_R(1)+\mu C(R)  \\
    &= m(a) +o_R(1)+\mu C(R).
\end{align*}
Let fixed $R$ be sufficiently large such that $o_R(1)\leq \frac{\varepsilon}{2}$, and take $\delta$ sufficiently small satisfying $\mu C(R) \leq \frac{\varepsilon}{2}$ for $\mu \in (0,\delta)$, thus we have
$$ m_\mu(a) \leq m(a) +\varepsilon.$$
\end{proof}

Now, we can get the critical point of $I_\mu|_{\mathcal{S}(a)}$.

\begin{lemma}\label{lem3.5'}
For any fixed $\mu \in (0,1]$, there exist a $u_\mu \in \Upsilon \backslash \{0\}$ and a $\lambda_\mu \in \mathbb{R}$ such that

(1) $I'_\mu(u_\mu) +\lambda_\mu u_\mu=0;$

(2) $P_\mu(u_\mu)=0;$

(3) $I_\mu(u_\mu)=m_\mu(a);$

(4) $\|u_\mu\|_2 = a.$
\end{lemma}

\begin{proof}
Let $u_n$ be a sequence in Lemma \ref{lem3.4.1}. Taking into account Lemma \ref{lem2.5}(2) and Lemma \ref{lem3.4.1}, we can conclude that $u_n$ is bounded in $\Upsilon$. Thus, up to a subsequence, there exists a $u_\mu \in \Upsilon$ such that
\begin{align*}
 &u_n \rightharpoonup u_\mu     \;\; \;\;\;\;\;\;\;\;\;\;\;\;\,  \mbox{in} \; \Upsilon , \\
&u_n \nabla u_n \rightharpoonup u \nabla u \;\;\;\;\,\mbox{in}   \;\; (L^2(\mathbb{R}^N))^N, \\
 &u_n \rightarrow  u_\mu        \;\;\;\;\;\,\;\;\;\;\;\;\;\;\;   \mbox{in} \;L^q_{loc}(\mathbb{R}^N)\; \mbox{for}\; \mbox{all}\; q \in (2,2^*], \\
 &u_n \rightarrow  u_\mu  \;\;\;\;\;\;\;\;\;\;\;\;\;\;  \,             \mbox{a.e.} \; \mbox{in} \; \mathbb{R}^N.
\end{align*}

\textbf{Case 1:} $u_\mu=0$. If
$$\limsup_{n \rightarrow \infty} \sup_{y \in \mathbb{R}^N} \int_{B_1(y)} |u_n|^2=0,$$
then $\int_{\mathbb{R}^N} |u_n|^p \rightarrow 0$ . Thus it yields from
\begin{align*}
\mu(1+\gamma_\theta) \int_{\mathbb{R}^N} |\nabla u_n|^{\theta} +  \int_{\mathbb{R}^N} |\nabla u_n|^2 + (N+2)\int_{\mathbb{R}^N} |u_n|^2 |\nabla u_n|^2
=P_\mu(u_n) +\gamma_p \int_{\mathbb{R}^N} |u_n|^p \rightarrow 0
\end{align*}
that $I_\mu(u_n) \rightarrow 0$, in contradiction with $m_\mu(a) >0$. Therefore,
$$\limsup_{n \rightarrow \infty} \sup_{y \in \mathbb{R}^N} \int_{B_1(y)} |u_n|^2>0,$$
We may assume, up to a subsequence, that $v_n:=u_n(\cdot+y_n) \rightharpoonup v_\mu \neq 0$. It follows that $I_\mu(v_n) \rightarrow m_\mu(a)$. By \cite[Lemma 3]{BL1983} and Lemma \ref{lem3.4.1}, we conclude that there exists a sequence $\lambda_n \in \mathbb{R}$ such that
\begin{equation}\label{eq3.7'}
I'_\mu(u_n) +\lambda_n u_n \rightarrow 0 \;\; \mbox{in} \; \Upsilon^*.
\end{equation}
It is easy to see that
\begin{equation}\label{eq3.7''}
I'_\mu(v_n) +\lambda_n v_n \rightarrow 0 \;\; \mbox{in} \; \Upsilon^*.
\end{equation}
Hence $\lambda_n=-\frac{1}{a^2}I'_\mu(v_n)[v_n]+o_n(1)$ is bounded in $\mathbb{R}$. We assume, up to subsequence, $\lambda_n \rightarrow \lambda_\mu$. Since $v_n$ is bounded, we conclude
\begin{equation}\label{eq3.8'}
I'_\mu(v_n) +\lambda_\mu v_n \rightarrow 0 \;\; \mbox{in} \; \Upsilon^*.
\end{equation}
As a result,
\begin{equation}\label{eq3.9''}
I'_\mu(v_\mu) +\lambda_\mu v_\mu = 0 \;\; \mbox{in} \; \Upsilon^*.
\end{equation}
It yields that $P_\mu(v_\mu)=0$ and $I_\mu(v_\mu)>0$.
Let $v_{1n}:=v_n-v_\mu$ and $b_n:=\|v_{1n}\|_2$, then $b_n \rightarrow b:=(a^2-\|v_\mu\|_2^2)^\frac{1}{2}$. It is easy to see that the following decomposition result for $\{v_n\}$
\begin{equation*}
I_\mu(v_n)=I_\mu(v_\mu)+I_\mu(v_{1n})+o_n(1)
\end{equation*}
and
\begin{equation*}
P_\mu(v_n)=P_\mu(v_\mu)+P_\mu(v_{1n})+o_n(1)
\end{equation*}
If $b >0$, then $b_n >0$ for $n$ large enough, it is easy to see that $P_\mu(v_{1n}) \rightarrow 0$, by uniqueness, there exist $s_n \rightarrow 0$ such that $s_n* v_{1n} \in \mathcal{P}_\mu(b_n)$. Therefore, by Lemma \ref{lem6.1'} and Remark \ref{rem7.2}, we conclude
\begin{align*}
    m_\mu(a) \leq m_\mu(b) &\leq \liminf_{n \rightarrow \infty}m_\mu(b_n)  \\
    &\leq \liminf_{n \rightarrow \infty} I_\mu(s_n*v_{1n})   \\
    &= \liminf_{n \rightarrow \infty}I_\mu(v_{1n})   \\
    &= \liminf_{n \rightarrow \infty}I_\mu(v_{n}) - I_\mu(v_{\mu})    \\
    &= m_\mu(a) - I_\mu(v_{\mu})    \\
    &< m_\mu(a),
\end{align*}
this is impossible. Thus $b=0$, which means $\|v_\mu\|_2=a$ and $v_n \rightarrow  v_\mu$ in $L^2(\mathbb{R}^N)$.
It follows from the interpolation inequality that we obtain
$$\;\;v_n \rightarrow  v_\mu        \;\;\;\;\;\;\;\;\;\;\;\;\;\;\;   \mbox{in} \;L^q(\mathbb{R}^N)\; \mbox{for}\; \mbox{all}\; q \in (2,2^*].$$
As a result,
$$ P_\mu(v_n) +\gamma_p \int_{\mathbb{R}^N} |v_n|^p  \rightarrow P_\mu(v_\mu) +\gamma_p \int_{\mathbb{R}^N} |v_\mu|^p.$$
Taking into account the weak lower semicontinuous property \cite{CJS2010}, we have that
\begin{equation}\label{eq3.10.1}
\mu \int_{\mathbb{R}^N} |\nabla v_n|^{\theta} \rightarrow \mu \int_{\mathbb{R}^N} |\nabla v_\mu|^{\theta},
\end{equation}
\begin{equation}\label{eq3.11.1}
 \int_{\mathbb{R}^N} |\nabla v_n|^2 \rightarrow \int_{\mathbb{R}^N} |\nabla v_\mu|^2,
\end{equation}
\begin{equation}\label{eq3.12.1}
 \int_{\mathbb{R}^N} |v_n|^2 |\nabla v_n|^2 \rightarrow \int_{\mathbb{R}^N} |v_\mu|^2 |\nabla v_\mu|^2.
\end{equation}
Therefore,
$$ I_\mu(v_\mu)=\lim_{n\rightarrow \infty} I_\mu(v_n)=m_\mu(a).$$
It is easy to see that $\{v_n\} \subset \mathcal{P}_\mu(a)$ is a Palais-Smale sequence  for $I_\mu$ constrained on $\mathcal{S}(a)$ at the
level $m_\mu(a)$. Without loss of generality, we can write $v_\mu$ as $u_\mu$, thus we complete the proof.

\textbf{Case 2:} $u_\mu \neq 0$.
By the similar approach to $v_n \rightharpoonup v_\mu \neq 0$, we can complete the proof.
\end{proof}

\begin{remark}
    We apply a different method to overcome the difficulty caused by the lack of compactness in non-radial space. If Lemmas \ref{lem6.1'} were valid for $p \in \left(4+\frac{4}{N},2\cdot2^*\right)$, we would obtain the existence
of ground state normalized solutions for $p \in \left(4+\frac{4}{N},2\cdot2^*\right)$.
\end{remark}

\begin{lemma}\label{lem4.1'}
Let $\mu_n \rightarrow 0^+$, there exist sequence $\lambda_{\mu_n} \geq 0$ and $u_{\mu_n} \in \mathcal{S}(a)$ such that $I'_{\mu_n}(u_{\mu_n})+\lambda_{\mu_n} u_{\mu_n}=0$, $I_{\mu_n}(u_{\mu_n})=m_{\mu_n}(a) \rightarrow c>0$. Then

(1) there exists a subsequence $u_{\mu_n}\rightharpoonup u$ in $H^1(\mathbb{R}^N)$ with $u\in H^1(\mathbb{R}^N) \cap L^{\infty}(\mathbb{R}^N)$;

(2) there exists a $\lambda \in \mathbb{R}$ such that $\Theta(u,\lambda)=0$;

(3) $I(u)=c$, $0<\|u\|_2\leq a$; if $\lambda\neq0$, then $\|u\|_2=a$.
\end{lemma}

\begin{proof}
(1) It follows from $I'_{\mu_n}(u_{\mu_n})+\lambda_{\mu_n} u_{\mu_n}=0$ that $P_{\mu_n}(u_{\mu_n})=0$ for any $n \in \mathbb{N}^+$. By Lemma \ref{lem2.5}(2), $I_{\mu_n}(u_{\mu_n}) \rightarrow c>0$ implies that $u_{\mu_n}$ is bounded in $H^1(\mathbb{R}^N)$ and $\int_{\mathbb{R}^N} |u_{\mu_n}|^2 |\nabla u_{\mu_n}|^2 \leq C$. Thus, up to a subsequence, there exists a $u \in H^1(\mathbb{R}^N)$ such that
\begin{align*}
 &u_{\mu_n} \rightharpoonup u     \;\; \;\;\;\;\;\;\;\;\;\;\;\;\;\;\;\;\;  \mbox{in} \; H^1(\mathbb{R}^N) , \\
 &u_{\mu_n} \nabla u_{\mu_n} \rightharpoonup u \nabla u \;\;\;\;\,\,\mbox{in}   \; (L^2(\mathbb{R}^N))^N, \\
 &u_{\mu_n} \rightarrow  u        \;\;\;\;\;\;\;\;\;\;\;\;\;\;\;\;\;\;\;   \mbox{in} \;L^q_{loc}(\mathbb{R}^N)\;\, \mbox{for}\; \mbox{all}\; q \in (2,2^*], \\
 &u_{\mu_n} \rightarrow  u    \;\;\;\;\;\;\;\;\;\;\;\;\;\;\;\;\;\;\;          \mbox{a.e.} \; \mbox{in} \; \mathbb{R}^N.
\end{align*}
At the same time, it is easy to see that $\lambda_{\mu_n}=-\frac{1}{a^2}I'_{\mu_n}(u_{\mu_n})[u_{\mu_n}]$ is bounded in $\mathbb{R}$. So, up to a subsequence, $\lambda_{\mu_n} \rightarrow \lambda$ in $ \mathbb{R}$.

{\bf Claim:} There exists a constant $C>0$ such that $\|u_{\mu_n}\|_{\infty} \leq C$ and $\|u\|_{\infty} \leq C$.

{\bf Case 1:} If $\|u_{\mu_n}\|_\infty \leq 1$, then Claim is obvious.

{\bf Case 2:} If Case 1 is impossible, let
\begin{align*}
\begin{split}
w_n:=
\left \{
      \begin{array}{ll}
           u_{\mu_n},  &|u_{\mu_n}| > 1, \\
           0,           &|u_{\mu_n}| \leq1. \\
      \end{array}
\right.
\end{split}
\end{align*}
Set $M>1, r>0$ and
\begin{align*}
v_n:=\max\{-M,\min\{w_n, M\}\}.
\end{align*}
Let $\varphi=w_n|v_n|^{2r}$, then $\varphi \in \Upsilon$. By $I'_{\mu_n}(u_{\mu_n})+\lambda_{\mu_n} u_{\mu_n}=0$, we can infer that
\begin{align*}
&\ \int_{\mathbb{R}^N}|w_n|^{p-2}w_n\varphi -\lambda_{\mu_n} \int_{\mathbb{R}^N}w_n\varphi   \\
=&\ \mu_n \int_{\mathbb{R}^N} |\nabla w_n|^{\theta-2}\nabla w_n \cdot \nabla \varphi
+\int_{\mathbb{R}^N} \nabla w_n \cdot \nabla \varphi
+2\int_{\mathbb{R}^N}\left(|w_n|^2 \nabla w_n \cdot \nabla \varphi +w_n\varphi|\nabla w_n|^2\right)  \\
\geq &\ 2\int_{\mathbb{R}^N} |w_n|^2\nabla w_n \cdot \nabla \varphi  \\
=&\ 2\int_{\mathbb{R}^N}\left( |w_n|^2|\nabla w_n|^2|v_n|^{2r} +2r|w_n|^2|v_n|^{2r-2}w_nv_n\nabla w_n\cdot \nabla v_n \right)  \\
=&\ \frac{1}{2} \int_{\mathbb{R}^N} \left||v_n|^r \nabla (|w_n|^2)\right|^2 +\frac{4}{r}\int_{\mathbb{R}^N}\left| |w_n|^2\nabla (|v_n|^r)\right|^2   \\
\geq&\ \frac{1}{r+4} \int_{\mathbb{R}^N} \left| \nabla (|w_n|^2|v_n|^r)\right|^2   \\
\geq&\ \frac{C}{(r+2)^2} \left[\int_{\mathbb{R}^N} \left(|w_n|^2|v_n|^r\right)^{2^*}\right]^{\frac{2}{2^*}}.
\end{align*}
It is easy to see that $\left||w_n|^{p-2} -\lambda_{\mu_n} \right| \leq C|w_n|^{p-2}$ in spt$(w_n)$. Taking into account the interpolation inequality, we get
\begin{align*}
&\int_{\mathbb{R}^N}|w_n|^{p-2}w_n\varphi -\lambda_{\mu_n} \int_{\mathbb{R}^N}w_n\varphi  \\
\leq & C\int_{\mathbb{R}^N} |w_n|^{p-2}w_n\varphi   =C\int_{\mathbb{R}^N} |w_n|^p|v_n|^{2r}  \\
\leq & C\left(\int_{\mathbb{R}^N} |w_n|^{2\cdot2^*}\right)^{\frac{p-4}{2\cdot2^*}} \left[\int_{\mathbb{R}^N} \left(|v_n|^r|w_n|^2\right)^{\frac{4\cdot2^*}{2\cdot2^*-p+4}}\right]^{\frac{2\cdot2^*-p+4}{2\cdot2^*}}  \\
\leq & C \left[\int_{\mathbb{R}^N} \left(|v_n|^r|w_n|^2\right)^{\frac{4\cdot2^*}{2\cdot2^*-p+4}}\right]^{\frac{2\cdot2^*-p+4}{2\cdot2^*}}.
\end{align*}
As a result,
\begin{equation}\label{eq4.1}
\left[\int_{\mathbb{R}^N} \left(|v_n|^r|w_n|^2\right)^{2^*}\right]^{\frac{1}{2^*}} \leq C(r+2) \left[\int_{\mathbb{R}^N} \left(|v_n|^r|w_n|^2\right)^{\frac{4\cdot2^*}{2\cdot2^*-p+4}}\right]^{\frac{2\cdot2^*-p+4}{4\cdot2^*}}.
\end{equation}
Let $r_0={2^*-\frac{p}{2}}$, $d=\frac{2\cdot2^*-p+4}{4}$. Taking $r=r_0$ in \eqref{eq4.1}, and letting $M \rightarrow +\infty$, we can derive
\begin{equation}\label{eq4.2}
\|w_n\|_{2\cdot2^*d} \leq \ [C(r_0+2)]^\frac{1}{r_0+2}\|w_n\|_{2\cdot2^*}.
\end{equation}
Set $r_{i+1}+2=(r_i+2)d$ for $i\in\mathbb{N}^+$. Then, by induction, we infer that
\begin{equation}\label{eq4.3}
\|w_n\|_{2\cdot2^*d^{i+1}} \leq \ \prod^i_{k=0} [C(r_k+2)]^\frac{1}{r_k+2}\|w_n\|_{2\cdot2^*}   \leq C\|w_n\|_{2\cdot2^*}.
\end{equation}
Taking $i \rightarrow +\infty$ in \eqref{eq4.3}, we obtain
$$\|w_n\|_{\infty} \leq C.$$
Therefore,
\begin{equation}\label{eq4.4}
\|u_{\mu_n}\|_{\infty} \leq C,\ \|u\|_{\infty} \leq C.
\end{equation}

(2) Let $\varphi=\psi e^{-u_{\mu_n}}$ with $0\leq \psi \in \mathcal{C}^{\infty}_0(\mathbb{R}^N)$, we can derive
\begin{align*}
0=&\ (I'_{\mu_n}(u_{\mu_n})+\lambda_{\mu_n} u_{\mu_n})[\varphi] \\
 =&\ \mu_n \int_{\mathbb{R}^N} |\nabla u_{\mu_n}|^{\theta-2}\nabla u_{\mu_n} \cdot (\nabla \psi e^{-u_{\mu_n}}-\psi e^{-u_{\mu_n}}\nabla u_{\mu_n})       +\int_{\mathbb{R}^N} \nabla u_{\mu_n} \cdot (\nabla \psi e^{-u_{\mu_n}}-\psi e^{-u_{\mu_n}}\nabla u_{\mu_n})  \\
 &+2\int_{\mathbb{R}^N} |u_{\mu_n}|^2 \nabla u_{\mu_n} \cdot (\nabla \psi e^{-u_{\mu_n}}-\psi e^{-u_{\mu_n}}\nabla u_{\mu_n})
 +2\int_{\mathbb{R}^N} u_{\mu_n}\psi e^{-u_{\mu_n}} |\nabla u_{\mu_n}|^2 +\lambda_{\mu_n} \int_{\mathbb{R}^N} u_{\mu_n}\psi e^{-u_{\mu_n}}  \\
 & -\int_{\mathbb{R}^N} |u_{\mu_n}|^{p-2}u_{\mu_n}\psi e^{-u_{\mu_n}}  \\
 \leq&\ \mu_n \int_{\mathbb{R}^N} |\nabla u_{\mu_n}|^{\theta-2}\nabla u_{\mu_n} \cdot \nabla \psi e^{-u_{\mu_n}}
 + \int_{\mathbb{R}^N}\left(1+2u_{\mu_n}^2\right)\nabla u_{\mu_n} \cdot \nabla \psi e^{-u_{\mu_n}}   \\
 &-\int_{\mathbb{R}^N}\left(1+2u_{\mu_n}^2-2u_{\mu_n}\right) \psi e^{-u_{\mu_n}}|\nabla u_{\mu_n}|^2   +\lambda_{\mu_n} \int_{\mathbb{R}^N} u_{\mu_n}\psi e^{-u_{\mu_n}}  \\
 &-\int_{\mathbb{R}^N} |u_{\mu_n}|^{p-2}u_{\mu_n}\psi e^{-u_{\mu_n}}.
\end{align*}
By Lemma \ref{lem2.5}(2) and $\mu_n \rightarrow 0^+$, we infer that
\begin{align*}
     \left|\mu_n \int_{\mathbb{R}^N} |\nabla u_{\mu_n}|^{\theta-2}\nabla u_{\mu_n} \cdot \nabla \psi e^{-u_{\mu_n}} \right|
\leq C|\mu_n \|\nabla u_{\mu_n}\|^{\theta-1}_\theta|
=    C|\mu_n \|\nabla u_{\mu_n}\|^{\theta}_\theta|^{\frac{\theta-1}{\theta}} |\mu_n|^{\frac{1}{\theta}}  \rightarrow 0^+.
\end{align*}
Taking into account Claim in (1), the H\"older inequality, the weak convergence and the Lebesgue's dominated convergence theorem, we can deduce that
$$\int_{\mathbb{R}^N}\left(1+2u_{\mu_n}^2\right)\nabla u_{\mu_n} \cdot \nabla \psi e^{-u_{\mu_n}} \rightarrow  \int_{\mathbb{R}^N}(1+2u^2)\nabla u \cdot \nabla \psi e^{-u} ,$$
$$\lambda_{\mu_n} \int_{\mathbb{R}^N} u_{\mu_n}\psi e^{-u_{\mu_n}}  \rightarrow \lambda \int_{\mathbb{R}^N} u\psi e^{-u} ,$$
$$\int_{\mathbb{R}^N} |u_{\mu_n}|^{p-2}u_{\mu_n}\psi e^{-u_{\mu_n}} \rightarrow \int_{\mathbb{R}^N} |u|^{p-2}u\psi e^{-u}.$$
By the Fatou's lemma, we conclude that
$$\liminf_{n\rightarrow \infty} \int_{\mathbb{R}^N}(1+2u_{\mu_n}^2-2u_{\mu_n}) \psi e^{-u_{\mu_n}}|\nabla u_{\mu_n}|^2 \geq \int_{\mathbb{R}^N}(1+2u^2-2u) \psi e^{-u}|\nabla u|^2. $$
As a result,
\begin{align*}
0\leq & \int_{\mathbb{R}^N} \nabla u \cdot (\nabla \psi e^{-u}-\psi e^{-u}\nabla u)  +2\int_{\mathbb{R}^N} |u|^2 \nabla u \cdot (\nabla \psi e^{-u}-\psi e^{-u}\nabla u)  \\
 & +2\int_{\mathbb{R}^N} u\psi e^{-u} |\nabla u|^2 +\lambda \int_{\mathbb{R}^N} u\psi e^{-u}
 -\int_{\mathbb{R}^N} |u|^{p-2}u\psi e^{-u} .
\end{align*}
For any fixed $\phi \in \mathcal{C}^{\infty}_0(\mathbb{R}^N)$ with $\phi \geq0$, we take $\psi_n \in \mathcal{C}^{\infty}_0(\mathbb{R}^N)$ and $\psi_n \geq0$ such that
\begin{align*}
\psi_n \rightarrow \phi e^u \;\mbox{in} \; H^1(\mathbb{R}^N),\;\;  \psi_n \rightarrow \phi e^u \; \mbox{a.e.} \;\mbox{in} \; \mathbb{R}^N.
\end{align*}
Therefore,
\begin{align*}
0\leq & \int_{\mathbb{R}^N} \nabla u \cdot \nabla \phi   +2\int_{\mathbb{R}^N} \left(|u|^2 \nabla u \cdot \nabla \phi
  +u\phi |\nabla u|^2\right) +\lambda \int_{\mathbb{R}^N} u\phi
 -\int_{\mathbb{R}^N} |u|^{p-2}u\phi .
\end{align*}
Similarly by taking $\varphi=\psi e^{u_{\mu_n}}$, we get
\begin{align*}
0\geq & \int_{\mathbb{R}^N} \nabla u \cdot \nabla \phi   +2\int_{\mathbb{R}^N} \left(|u|^2 \nabla u \cdot \nabla \phi
  +u\phi |\nabla u|^2\right) +\lambda \int_{\mathbb{R}^N} u\phi
 -\int_{\mathbb{R}^N} |u|^{p-2}u\phi .
\end{align*}
It is easy to get
\begin{equation}\label{eq4-9}
   \Theta(u,\lambda)=0,
\end{equation}
which implies that $P(u)=0.$

(3) \textbf{Case 1:} $u=0$. If
$$\limsup_{n \rightarrow \infty} \sup_{y \in \mathbb{R}^N} \int_{B_1(y)} |u_{\mu_n}|^2=0,$$
then $\int_{\mathbb{R}^N} |u_{\mu_n}|^p \rightarrow 0$ . Thus it yields from
\begin{align*}
\mu_n(1+\gamma_\theta) \int_{\mathbb{R}^N} |\nabla u_{\mu_n}|^{\theta} +  \int_{\mathbb{R}^N} |\nabla u_{\mu_n}|^2 + (N+2)\int_{\mathbb{R}^N} |u_{\mu_n}|^2 |\nabla u_{\mu_n}|^2
=P_{\mu_n}(u_{\mu_n}) +\gamma_p \int_{\mathbb{R}^N} |u_{\mu_n}|^p \rightarrow 0
\end{align*}
that $I_{\mu_n}(u_{\mu_n}) \rightarrow 0$, in contradiction with $c >0$. Therefore,
$$\limsup_{n \rightarrow \infty} \sup_{y \in \mathbb{R}^N} \int_{B_1(y)} |u_{\mu_n}|^2>0,$$
We may assume, up to a subsequence, that $v_{\mu_n}:=u_{\mu_n}(\cdot+y_n) \rightharpoonup v \neq 0$.
As a result,
\begin{equation}\label{eq3.9'}
\Theta(v,\lambda) = 0.
\end{equation}
It yields that $P(v)=0$ and $I(v)>0$. Moreover, denote $v_{1n}=v_{\mu_n}-v$,
if
$$\limsup_{n \rightarrow \infty} \sup_{y \in \mathbb{R}^N} \int_{B_1(y)} |v_{1n}|^2>0,$$
We may assume, up to a subsequence, that $v_{1n}(\cdot+y_{1n}) \rightharpoonup v_1 \neq 0$.
As a result,
\begin{equation}\label{eq3.9.1}
\Theta(v_1,\lambda) = 0.
\end{equation}
It yields that $P(v_1)=0$ and $I(v_1)>0$. Moreover, denote $v_{2n}=v_{1n}-v_1(\cdot-y_{1n})$, then, similar to \cite[Corollary A.2]{GG2024}, we can prove
\begin{equation*}
I_{\mu_n}(v_{\mu_n})=\frac{\mu_n}{\theta} \int_{\mathbb{R}^N} |\nabla v_{\mu_n}|^{\theta}+I(v)+I(v_{1})+I(v_{2n})+o_n(1),
\end{equation*}
and
\begin{equation*}
P_{\mu_n}(v_{\mu_n})=\mu_n(1+\gamma_\theta) \int_{\mathbb{R}^N} |\nabla v_{\mu_n}|^{\theta}+P(v)+P(v_{1})+P(v_{2n})+o_n(1).
\end{equation*}
It follows from  $P(v_j)=0$ and
\begin{equation*}
C\geq \int_{\mathbb{R}^N} |v_{\mu_n}|^2 |\nabla v_{\mu_n}|^2=\int_{\mathbb{R}^N} |v|^2 |\nabla v|^2+\int_{\mathbb{R}^N} |v_1|^2 |\nabla v_1|^2+\int_{\mathbb{R}^N} |v_{2n}|^2 |\nabla v_{2n}|^2+o_n(1)
\end{equation*}
that
\begin{align*}
(N+2)\int_{\mathbb{R}^N} |v_j|^2 |\nabla v_j|^2 \leq  \gamma_p \int_{\mathbb{R}^N} |v_j|^p \leq  C_p\left (\int_{\mathbb{R}^N} |v_j|^2\right)^{\frac{4N-p(N-2)}{2(N+2)}}\left(4\int_{\mathbb{R}^N} |v_j|^2 |\nabla v_j|^2\right)^{\frac{N(p-2)}{2(N+2)}},
\end{align*}
which implies
$$\int_{\mathbb{R}^N} |v_j|^2 \geq C_{p,N}.$$
At the same time, since
\begin{equation*}
\int_{\mathbb{R}^N} |v_{\mu_n}|^2=\int_{\mathbb{R}^N} |v|^2 +\int_{\mathbb{R}^N} |v_1|^2 +\int_{\mathbb{R}^N} |v_{2n}|^2 +o_n(1),
\end{equation*}
we get this procedure will be end after finite times. We obtain
the existence of $v_j$ with $\|v_j\|_2
=a_j>0$ and
\begin{equation*}
I_{\mu_n}(v_{\mu_n})=\frac{\mu_n}{\theta} \int_{\mathbb{R}^N} |\nabla v_{\mu_n}|^{\theta}+I(v)+\sum_{j=1}^k I(v_{j})+o_n(1),
\end{equation*}
and
\begin{equation*}
P_{\mu_n}(v_{\mu_n})=\mu_n(1+\gamma_\theta) \int_{\mathbb{R}^N} |\nabla v_{\mu_n}|^{\theta}+P(v)+\sum_{j=1}^k P(v_{j})+o_n(1).
\end{equation*}
Since $P(v_j)=0$, we get $\mu_n \int_{\mathbb{R}^N} |\nabla v_{\mu_n}|^{\theta}=o_n(1)$, which yields
\begin{align*}
&\ m(a) + \varepsilon   \\
\geq &\ m_{\mu_n}(a)=I_{\mu_n}(v_{\mu_n})
  \\
  = &\ I(v)+\sum_{j=1}^k I(v_{j})+o_n(1) \\
  \geq &\ I(v)+I(v_{1}) +o_n(1) \\
  \geq&\ I(v)+m(\|v_1\|_2),
\end{align*}
this is impossible.

Thus
$$\limsup_{n \rightarrow \infty} \sup_{y \in \mathbb{R}^N} \int_{B_1(y)} |v_{1n}|^2=0,$$
then $\int_{\mathbb{R}^N}|v_{1n}|^p \rightarrow 0$. As a result,
$$ P_{\mu_n}(v_{\mu_n}) +\gamma_p \int_{\mathbb{R}^N} |v_{\mu_n}|^p  \rightarrow P(v) +\gamma_p \int_{\mathbb{R}^N} |v|^p.$$
Taking into account the weak lower semicontinuous property \cite{CJS2010}, we have that
\begin{equation}\label{eq3.10'}
\mu_n \int_{\mathbb{R}^N} |\nabla v_{\mu_n}|^{\theta} \rightarrow 0,
\end{equation}
\begin{equation}\label{eq3.11'}
 \int_{\mathbb{R}^N} |\nabla v_{\mu_n}|^2 \rightarrow \int_{\mathbb{R}^N} |\nabla v|^2,
\end{equation}
\begin{equation}\label{eq3.12'}
 \int_{\mathbb{R}^N} |v_{\mu_n}|^2 |\nabla v_{\mu_n}|^2 \rightarrow \int_{\mathbb{R}^N} |v|^2 |\nabla v|^2.
\end{equation}
Therefore,
$$ I(v)=\lim_{n\rightarrow \infty} I_{\mu_n}(v_{\mu_n})=c.$$
Consequently,
$$0<\|v\|_2 \leq \liminf_{n\rightarrow \infty} \|v_{\mu_n}\|_2 = a.$$
It follows from \eqref{eq3.10'}$-$\eqref{eq3.12'} that
\begin{equation}\label{eq4.13}
    I'_{\mu_n}(v_{\mu_n})[v_{\mu_n}] \rightarrow \int_{\mathbb{R}^N} |\nabla v|^2  +4\int_{\mathbb{R}^N} |v|^2 |\nabla v|^2
 -\int_{\mathbb{R}^N} |v|^p.
\end{equation}
 And it is easy to see that
\begin{equation}\label{eq4-4'}
I'_{\mu_n}(v_{\mu_n}) +\lambda v_{\mu_n} \rightarrow 0 \;\; \mbox{in} \; \Upsilon^*.
\end{equation}
Thus combining \eqref{eq3.9'} with \eqref{eq4.13} and \eqref{eq4-4'}, we obtain
$\lambda \|v_{\mu_n}\|^2_2 \rightarrow \lambda \|v\|^2_2$.
If $\lambda \neq0$, then
$$\|v\|_2=a.$$
Without loss of generality, we can write $v$ as $u$, thus we complete the proof.

\textbf{Case 2:} $u \neq 0$.
By the similar approach to $v_{\mu_n} \rightharpoonup v \neq 0$, we can complete the proof.
\end{proof}

\begin{remark}
    Since $\mu_n \rightarrow 0^+$,
    the method of getting compactness in Lemma \ref{lem3.5'} is invalid. Thus, we have to apply a different one to deal with it.
\end{remark}

\begin{proof} [Proof of Theorem \ref{Thm1.1'}]
By Remark \ref{rem02.5}, we can conclude that for any $0<\mu_1<\mu_2 \leq1$,
\begin{align*}
    m_{\mu_1}(a)=& \inf_{u \in \mathcal{P}_{\mu_1}(a)}I_{\mu_1}(u)   \\
    =& \inf_{u \in \mathcal{S}(a)}\max_{s \in \mathbb{R}}I_{\mu_1}(s*u)   \\
    \leq& \inf_{u \in \mathcal{S}(a)}\max_{s \in \mathbb{R}}I_{\mu_2}(s*u)   \\
    =& \inf_{u \in \mathcal{P}_{\mu_2}(a)}I_{\mu_2}(u)= m_{\mu_2}(a).
\end{align*}
It follows that $m_\mu(a)$ is nondecreasing with respect to $\mu \in (0,1]$. From Lemma \ref{lem2.5}, we infer that
$$b(a):=\lim_{\mu \rightarrow 0^+} m_\mu(a) \geq \frac{\zeta}{8}>0.$$
It follows from Lemma \ref{lem3.5'} that there exists a sequence $\{\mu_n\}$ such that
$$ I'_{\mu_n}(u_{\mu_n})+\lambda_{\mu_n}u_{\mu_n}=0, \;\; I_{\mu_n}(u_{\mu_n})=m_{\mu_n}(a) \rightarrow b(a), \;\mbox{as} \; \mu_n \rightarrow 0^+,$$
where $u_{\mu_n} \in \mathcal{S}(a)$. At the same time, Lemma \ref{lem4.1'} implies that there exists a $v \in H^1(\mathbb{R}^N)\cap L^{\infty}(\mathbb{R}^N)$ and a $\lambda_0 \in \mathbb{R}$ such that
$$\Theta(v,\lambda_0)=0, \;\; I(v)=b(a),\;\; 0<\|v\|_2\leq a.$$
So, by Lemma \ref{lambda}, we get $\lambda_0>0$, which means $\|v\|_2=a$. Obviously, $v$ is a nontrivial solution of \eqref{eq01} and \eqref{eq02}.

Recall
$$\sigma(a)=\inf_{u \in \mathcal{S}'(a), \Theta(u,\lambda)=0} I(u).$$
Obviously, $\sigma(a) \leq I(v)=b(a)$. Further, using similar approach to Lemma \ref{lem2.5}(1), we conclude that $\sigma(a)>0$.

We take a sequence $v_n \neq0$ such that
$$ v_n \in \mathcal{S}'(a),\;  \Theta(v_n,\lambda_n)=0,\;  I(v_n) \rightarrow \sigma(a).$$
By the similar approach to Lemma \ref{lem4.1'}, we conclude that
there exists a $u \neq0,\; u \in H^1(\mathbb{R}^N)\cap L^{\infty}(\mathbb{R}^N)$ and a $\lambda \in \mathbb{R}$ such that
$$\Theta(u,\lambda)=0, \;\; I(u)=\sigma(a).$$
It follows from Lemma \ref{lambda} that $\lambda >0$. As a result, $\|u\|_2=a$.
\end{proof}

\section{The case of radial space }\label{radial space}

In this section, we assume $N=3,4$ and $p \in \left(4+\frac{4}{N}, 2\cdot2^*\right)$, we study $\Psi_\mu$ on the radial space $\mathbb{R} \times \mathcal{S}_r(a)$ with
$$\mathcal{S}_r(a):=\mathcal{S}(a)\cap \Upsilon_{r}.$$
It is easy to see that $\Psi_\mu$ is of class $\mathcal{C}^1$. By the symmetric critical point principle \cite{P1979}, a Palais-Smale sequence for $\Psi_\mu|_{\mathbb{R} \times \mathcal{S}_r(a)}$ is also a Palais-Smale sequence for $\Psi_\mu|_{\mathbb{R} \times \mathcal{S}(a)}$.

Let $Y \subset \Upsilon$. A set $A \subset Y$ is called $\tau$-invariant if $\tau(A)=A$. A homotopy $\eta:[0,1]\times Y \mapsto Y$ is $\tau$-equivariant if $\eta(t,\tau(u))=\tau(\eta(t,u))$ for any $(t,u) \in [0,1]\times Y$.

\begin{definition}\label{def5.1}
\cite[Definition 7.1]{G1993} Let B be a closed $\tau$-invariant subset of Y. We say that a class $\mathcal{G}$ of compact subsets of Y is a $\tau$-homotopy stable family with boundary B provided

(a) every set in $\mathcal{G}$ is $\tau$-invariant;

(b) every set in $\mathcal{G}$ contains B;

(c) for any set A in $\mathcal{G}$ and any $\tau$-equivariant homoyopy $\eta \in \mathcal{C}([0,1]\times Y,Y)$ satisfying $\eta(t,x)=x$ for all $(t,x)$ in $(\{0\}\times Y)\cup ([0,1]\times B) $ we have that $\eta(1,A) \in \mathcal{G}.$
\end{definition}

 Taking $\tau(u)=-u$, then $\Psi_\mu(u)$ is $\tau$-invariant. Similar to Lemma \ref{lem3.4.1} and by \cite[Theorem 7.2]{G1993}, we have the next lemma.

\begin{lemma}\label{lem5.3}
Let $\mathcal{G}$ is a $\tau$-homotopy stable family of compact subsets of $\mathcal{S}_r(a)$ with boundary $B= \emptyset$, and set
$$E_\mu^\mathcal{G}(a):=\inf_{A \in \mathcal{G}} \max_{u \in A} \Psi_\mu(u).$$
If $E_\mu^\mathcal{G}(a)>0$, then there exists a sequence $u_n \in \mathcal{S}_r(a)$ such that
$$I_\mu(u_n) \rightarrow E_\mu^\mathcal{G}(a),\; I_\mu|'_{\mathcal{S}(a)}(u_n) \rightarrow 0,\; P_\mu(u_n)=0.$$
\end{lemma}

We recall the definition of the genus of $\tau$-invariant sets due to M. A. Krasnoselskii and refer the readers to \cite{R1986}.
\begin{definition}
For any nonempty closed $\tau$-invariant set $A \subset \Upsilon$, the genus of A is defined by
\begin{align*}
    {\rm gen}(A):=\min\{k \in \mathbb{N}^+\;| \;\exists\, \phi:A \mapsto \mathbb{R}^k \backslash \{0\}, \phi \;is \; odd \; and \; continuous\}.
\end{align*}
\end{definition}
We take
\begin{align*}
  \begin{split}
     \mbox{gen}(A)=\left \{
        \begin{array}{ll}
            +\infty,    &\mbox{finite} \; k \; \mbox{does} \; \mbox{not}\; \mbox{exist}, \\
            0,          &A =\emptyset.
        \end{array}
     \right.
  \end{split}
\end{align*}
Let $\mathcal{A}(a)$ be the family of compact $\tau$-invariant subsets of $\mathcal{S}_r(a)$. For any $j \in \mathbb{N}^+$, we define
$$\mathcal{A}_j(a):=\{A \in \mathcal{A}(a)\;|\; \mbox{gen}(A) \geq j\}$$
and
$$c^j_\mu(a):=\inf_{A \in \mathcal{A}_j(a)} \max_{u \in A} \Psi_\mu(u).$$
By a similar approach to  \cite[Lemma 3.13]{LZ2023}, we obtain the next lemma.
\begin{lemma}\label{lem5.5}
(1) $\mathcal{A}_j(a) \neq \emptyset$ for any $j\in \mathbb{N}^+$, and $\mathcal{A}_j(a)$ is a $\tau$-homotopy stable family of compact subsets of $\mathcal{S}_r(a)$ with boundary $B = \emptyset$;

(2) $c^{j+1}_\mu(a) \geq c^j_\mu(a) \geq \frac{\zeta}{8} >0$ for any $\mu \in (0,1]$ and $j\in \mathbb{N}^+$;

(3) $c^j_\mu(a)$ is nondecreasing with respect to $\mu \in (0,1]$ for any $j\in \mathbb{N}^+$;

(4) $\tilde{b}^j(a):= \inf_{0<\mu \leq 1} c^j_\mu(a) \rightarrow +\infty$ as $j \rightarrow +\infty$.
\end{lemma}

Now, we can get the critical point of $I_\mu|_{\mathcal{S}(a)}$.

\begin{lemma}\label{lem5.6}
For any fixed $\mu \in (0,1]$ and any $j\in \mathbb{N}^+$, there exist a  $u^j_\mu \in \Upsilon_r \backslash \{0\}$ and a $\lambda^j_\mu \in \mathbb{R}$ such that

(1) $I'_\mu(u^j_\mu) +\lambda^j_\mu u^j_\mu=0;$

(2) $P_\mu(u^j_\mu)=0;$

(3) $I_\mu(u^j_\mu)=c^j_\mu(a);$

(4) $0<\|u^j_\mu\|_2 \leq a;$ furthermore, if $\lambda^j_\mu \neq 0$, then $\|u^j_\mu\|_2 = a.$
\end{lemma}

\begin{proof}
(1) By Lemma \ref{lem5.3} and Lemma \ref{lem5.5}, we may get a sequence $u_n^j \in \mathcal{S}_r(a)$ such that
$$I_\mu(u_n^j) \rightarrow c^j_\mu(a),\; I_\mu|'_{\mathcal{S}(a)}(u_n^j) \rightarrow 0,\;  P_\mu(u_n^j)=0,$$
as $n \rightarrow +\infty$. Taking into account Lemma \ref{lem2.5}(2), we can conclude that $u_n^j$ is bounded in $\Upsilon_r$. Thus, up to a subsequence, there exists a $u_\mu^j \in \Upsilon_r$ such that
\begin{align*}
 &u_n^j \rightharpoonup u_\mu^j     \;\; \;\;\;\;\;\;\;\;\;\;\;\;\;\;\;\;\,  \mbox{in} \; \Upsilon_r , \\
&u_n^j \nabla u_n^j \rightharpoonup u_\mu^j \nabla u_\mu^j \;\;\;\;\,\mbox{in}   \;\; (L^2(\mathbb{R}^N))^N, \\
 &u_n^j \rightarrow  u_\mu^j   \;\;\;\;\;\;\;\;\;\;\;\;\;\,\;\;\;\;\;   \mbox{in} \;L^q(\mathbb{R}^N)\; \mbox{for}\; \mbox{all}\; q \in (2,2^*), \\
 &u_n^j \rightarrow  u_\mu^j  \;\;\;\;\;\;\;\;\;\;\;\;\;\;\;\;\;\;  \,             \mbox{a.e.} \; \mbox{in} \; \mathbb{R}^N.
\end{align*}
It follows from \eqref{GNI} and the interpolation inequality that we obtain
$$\;\;u_n^j \rightarrow  u_\mu^j        \;\;\;\;\;\;\;\;\;\;\;\;\;\;\;   \mbox{in} \;L^q(\mathbb{R}^N)\; \mbox{for}\; \mbox{all}\; q \in (2,2\cdot2^*).$$
Assume $u_\mu^j=0$, then
\begin{align*}
\mu(1+\gamma_\theta) \int_{\mathbb{R}^N} |\nabla u_n^j|^{\theta} +  \int_{\mathbb{R}^N} |\nabla u_n^j|^2 + (N+2)\int_{\mathbb{R}^N} |u_n^j|^2 |\nabla u_n^j|^2
=P_\mu(u_n^j) +\gamma_p \int_{\mathbb{R}^N} |u_n^j|^p \rightarrow 0,
\end{align*}
which implies that $I_\mu(u_n^j) \rightarrow 0$, in contradiction with $c_\mu^j(a) >0$. Thus, $u_\mu^j \neq0.$
By \cite[Lemma 3]{BL1983}, we conclude that there exists a sequence $\lambda_n^j \in \mathbb{R}$ such that
\begin{equation}\label{eq3.7}
I'_\mu(u_n^j) +\lambda_n^j u_n^j \rightarrow 0 \;\; \mbox{in} \; \Upsilon^*.
\end{equation}
Hence $\lambda_n^j=-\frac{1}{a^2}I'_\mu(u_n^j)[u_n^j]+o_n(1)$ is bounded in $\mathbb{R}$. We assume, up to subsequence, $\lambda_n^j \rightarrow \lambda_\mu^j$. Since $u_n^j$ is bounded, we conclude
\begin{equation}\label{eq3.8}
I'_\mu(u_n^j) +\lambda_\mu^j u_n^j \rightarrow 0 \;\; \mbox{in} \; \Upsilon^*.
\end{equation}
As a result,
\begin{equation}\label{eq3.9}
I'_\mu(u_\mu^j) +\lambda_\mu^j u_\mu^j = 0 \;\; \mbox{in} \; \Upsilon^*.
\end{equation}

(2) By \eqref{eq3.9}, it is easy to see that $P_\mu(u_\mu^j)=0$.

(3) It follows from (2) that
$$ P_\mu(u_n^j) +\gamma_p \int_{\mathbb{R}^N} |u_n^j|^p  \rightarrow P_\mu(u_\mu^j) +\gamma_p \int_{\mathbb{R}^N} |u_\mu^j|^p  .$$
Taking into account the weak lower semicontinuous property \cite{CJS2010}, we have that
\begin{equation}\label{eq3.10}
\mu \int_{\mathbb{R}^N} |\nabla u_n^j|^{\theta} \rightarrow \mu \int_{\mathbb{R}^N} |\nabla u_\mu^j|^{\theta},
\end{equation}
\begin{equation}\label{eq3.11}
 \int_{\mathbb{R}^N} |\nabla u_n^j|^2 \rightarrow \int_{\mathbb{R}^N} |\nabla u_\mu^j|^2,
\end{equation}
\begin{equation}\label{eq3.12}
 \int_{\mathbb{R}^N} |u_n^j|^2 |\nabla u_n^j|^2 \rightarrow \int_{\mathbb{R}^N} |u_\mu^j|^2 |\nabla u_\mu^j|^2.
\end{equation}
Therefore,
$$ I_\mu(u_\mu^j)=\lim_{n\rightarrow \infty} I_\mu(u_n^j)=c_\mu^j(a).$$

(4) By \eqref{eq3.10}$-$\eqref{eq3.12}, we get
\begin{equation}\label{eq3.13}
I'_\mu(u_n^j)[u_n^j] \rightarrow I'_\mu(u_\mu^j)[u_\mu^j].
\end{equation}
Thus combining \eqref{eq3.8} with \eqref{eq3.9} and \eqref{eq3.13}, we conclude $\lambda_\mu^j \|u_n^j\|^2_2 \rightarrow \lambda_\mu^j \|u_\mu^j\|^2_2.$
Therefore, if $\lambda_\mu^j \neq0$, then
\begin{equation}\label{eq3.14}
 \|u_\mu^j\|_2=a.
\end{equation}

\end{proof}

\noindent Next, we establish a technical lemma, which means that the sequence of critical points of $I_\mu|_{\mathcal{S}(a)}$ converges to critical point $I|_{\mathcal{S}(a)}$ as $\mu \rightarrow 0^+$.

\begin{lemma}\label{lem4.1}
Let $\mu_n \rightarrow 0^+$, there exist sequence $\lambda_{\mu_n} \geq 0$ and $u_{\mu_n} \in \mathcal{S}_r(a_n)$ with $0<a_n\leq a$ such that $I'_{\mu_n}(u_{\mu_n})+\lambda_{\mu_n} u_{\mu_n}=0$, $I_{\mu_n}(u_{\mu_n}) \rightarrow c>0$. Then

(1) there exists a subsequence $u_{\mu_n}\rightharpoonup u$ in $H^1(\mathbb{R}^N)$ with $u\in H^1_r(\mathbb{R}^N) \cap L^{\infty}(\mathbb{R}^N)$;

(2) there exists a $\lambda \in \mathbb{R}$ such that $\Theta(u,\lambda)=0$;

(3) $I(u)=c$, $0<\|u\|_2\leq a$; if $\lambda\neq0$, then $\|u\|_2=\lim_{n\rightarrow \infty}a_n$.
\end{lemma}

\begin{proof}
(1) and (2) are similar to Lemma \ref{lem4.1'}, we omit it.

(3) Up to a subsequence, there exists a $u \in H^1(\mathbb{R}^N)$ such that
\begin{align*}
 &u_{\mu_n} \rightharpoonup u     \;\; \;\;\;\;\;\;\;\;\;\;\;\;\;\;\;\;\;  \mbox{in} \; H^1_r(\mathbb{R}^N) , \\
 &u_{\mu_n} \nabla u_{\mu_n} \rightharpoonup u \nabla u \;\;\;\;\,\,\mbox{in}   \; (L^2(\mathbb{R}^N))^N, \\
 &u_{\mu_n} \rightarrow  u        \;\;\;\;\;\;\;\;\;\;\;\;\;\;\;\;\;\;\;   \mbox{in} \;L^q(\mathbb{R}^N)\;\, \mbox{for}\; \mbox{all}\; q \in (2,2^*), \\
 &u_{\mu_n} \rightarrow  u   \;\;\;\;\;\;\;\;\;\;\;\;\;\;\;\;\;\;\;          \mbox{a.e.} \; \mbox{in} \; \mathbb{R}^N.
\end{align*}
Taking into account \eqref{GNI} and the interpolation inequality, we obtain that
$$\;u_{\mu_n} \rightarrow  u        \;\;\;\;\;\;\;\;\; \,  \mbox{in} \;L^q(\mathbb{R}^N)\;\, \mbox{for}\; \mbox{all}\; q \in (2,2\cdot2^*).$$
Assume $a_n \rightarrow 0$, it follows from \eqref{GNI} that $ \int_{\mathbb{R}^N} |u_{\mu_n}|^p \rightarrow 0$. Taking into account $P_{\mu_n}(u_{\mu_n})=0$, we can conclude that $I_{\mu_n}(u_{\mu_n}) \rightarrow 0<c$. This is impossible. Hence, $\liminf_{n\rightarrow \infty}a_n>0$, and so $\lambda_{\mu_n}=-\frac{1}{a^2_n}I'_{\mu_n}(u_{\mu_n})[u_{\mu_n}]$ is bounded in $\mathbb{R}$. Thus, up to a subsequence, $\lambda_{\mu_n} \rightarrow \lambda$ in $ \mathbb{R}$. And it is easy to see that
\begin{equation}\label{eq4-4}
I'_{\mu_n}(u_{\mu_n}) +\lambda u_{\mu_n} \rightarrow 0 \;\; \mbox{in} \; \Upsilon^*.
\end{equation}
Obviously,
$$P_{\mu_n}(u_{\mu_n}) +\gamma_p \int_{\mathbb{R}^N}|u_{\mu_n}|^p \rightarrow P(u) +\gamma_p\int_{\mathbb{R}^N}|u|^p.$$
Similar to \eqref{eq3.10}$-$\eqref{eq3.12}, we have
\begin{equation}\label{eq4.5}
\mu_n \int_{\mathbb{R}^N} |\nabla u_{\mu_n}|^{\theta} \rightarrow 0,
\end{equation}
\begin{equation}\label{eq4.6}
 \int_{\mathbb{R}^N} |\nabla u_{\mu_n}|^2 \rightarrow \int_{\mathbb{R}^N} |\nabla u|^2,
\end{equation}
\begin{equation}\label{eq4.7}
 \int_{\mathbb{R}^N} |u_{\mu_n}|^2 |\nabla u_{\mu_n}|^2 \rightarrow \int_{\mathbb{R}^N} |u|^2 |\nabla u|^2.
\end{equation}
Then
$$I(u)=\lim_{n\rightarrow \infty} I_{\mu_n}(u_{\mu_n})=c.$$
Consequently,
$$0<\|u\|_2 \leq \liminf_{n\rightarrow \infty} \|u_{\mu_n}\|_2 \leq a.$$
It follows from \eqref{eq4.5}$-$\eqref{eq4.7} that
\begin{equation}\label{eq4.13.1}
    I'_{\mu_n}(u_{\mu_n})[u_{\mu_n}] \rightarrow \int_{\mathbb{R}^N} |\nabla u|^2  +4\int_{\mathbb{R}^N} |u|^2 |\nabla u|^2
 -\int_{\mathbb{R}^N} |u|^p.
\end{equation}
Thus combining \eqref{eq4-4} with $\Theta(u,\lambda)=0$ and \eqref{eq4.13.1}, we obtain
$\lambda \|u_{\mu_n}\|^2_2 \rightarrow \lambda \|u\|^2_2$.
If $\lambda \neq0$, then
$$\|u\|_2=\lim_{n\rightarrow \infty}a_n.$$

\end{proof}

Now, we can complete the proof of Theorem \ref{Thm1.2}.

\begin{proof}[Proof of Theorem \ref{Thm1.2}]
It follows from Lemma \ref{lem5.5} that
$$\tilde{b}^j(a)= \lim_{\mu \rightarrow 0^+} c^j_\mu(a)\geq \frac{\zeta}{8}>0,\; \tilde{b}^j(a) \rightarrow +\infty,$$
as $j \rightarrow +\infty$. By Lemma \ref{lem5.6}, for any $j\in \mathbb{N}^+$, there exists a sequence $\{\mu^j_n\}$ such that
$$\mu^j_n \rightarrow 0^+,\; I'_{\mu^j_n}\left(u^j_{\mu^j_n}\right) +\lambda^j_{\mu^j_n} u^j_{\mu^j_n}=0,\; I_{\mu^j_n}\left(u^j_{\mu^j_n}\right)=c^j_{\mu^j_n}(a) \rightarrow \tilde{b}^j(a),$$
as $n \rightarrow +\infty$, where $u^j_{\mu^j_n} \in \mathcal{S}_r(a^j_n)$ with $0<a^j_n \leq a$. As a result, Lemma \ref{lem4.1} implies that there exist a $u^j \in H^1_r(\mathbb{R}^N) \cap L^\infty(\mathbb{R}^N)$ and a $\lambda^j \in \mathbb{R}$ such that
$$\Theta(u^j,\lambda^j)=0,\; I(u^j)=\tilde{b}^j(a),\; 0<\|u^j\|_2\leq a.$$
It follows from Lemma \ref{lambda} that $\lambda^j >0$. So $\lambda^j_{\mu^j_n}>0$ for $n$ sufficiently large, which means $a^j_n=\left\|u^j_{\mu^j_n}\right\|_2=a$. Therefore, $\|u^j\|_2=\lim_{n\rightarrow \infty}a^j_n =a $. Moreover, $I(u^j)=\tilde{b}^j(a) \rightarrow +\infty$.
\end{proof}

\section{Proof of Theorem \ref{Thm1.5}}\label{sec6}

In this section, we assume $N\geq 3$ and split the proof of Theorem \ref{Thm1.5} into several Lemmas.

\begin{lemma}\label{lem6.3}
    If $4+\frac{4}{N}<p< 2\cdot2^*$, then the function $a \mapsto m(a)$ is lower semicontinuous with respect to any $a>0$.
\end{lemma}

\begin{proof}
It suffices to prove that for any positive sequence $\{a_n\}$ with $a_n \rightarrow a$ as $n \rightarrow \infty$, we have $m(a) \leq \liminf_{n \rightarrow \infty}m(a_n)$. We divide the proof into two steps.

\textbf{Step 1:}
We prove that $\limsup_{n\rightarrow \infty}m(a_n) \leq C$.

Fix $u \in \mathcal{S}'(a) \cap L^{\infty}(\mathbb{R}^N)$, set $u_n:=\frac{a_n}{a}u  \in \mathcal{S}'(a_n)$. It is clear that $\int_{\mathbb{R}^N} |u_n|^2 |\nabla u_n|^2 \rightarrow \int_{\mathbb{R}^N} |u|^2 |\nabla u|^2$ and $u_n \rightarrow u$ in $H^1(\mathbb{R}^N)$. By Lemma \ref{lem2.4}, there exists a unique $s(u_n) \in \mathbb{R}$ such that $s(u_n)*u_n \in \mathcal{P}'(a_n)$ and $\lim_{n \rightarrow \infty}s(u_n) = s(u).$
And thus
 $$s(u_n)*u_n \rightarrow s(u)*u\;\,\mbox{in} \; H^1(\mathbb{R}^N),\;\int_{\mathbb{R}^N} |s(u_n)*u_n|^2 |\nabla (s(u_n)*u_n)|^2 \rightarrow \int_{\mathbb{R}^N} |s(u)*u|^2 |\nabla (s(u)*u)|^2.$$
As a consequence,
$$\limsup_{n\rightarrow \infty}m(a_n) \leq
\limsup_{n\rightarrow \infty}I(s(u_n)*u_n) =
I(s(u)*u) .$$

\textbf{Step 2:}
We show that $m(a)  \leq \liminf_{n\rightarrow \infty}m(a_n)$.

In fact, for each $n \in \mathbb{N}^+$, from the definition of the infimum, there exists a sequence $v_n \in \mathcal{P}'(a_n)$ such that
$$I(v_n) \leq m(a_n)+\frac{1}{n}.$$
Let $t_n:=\left(\frac{a}{a_n}\right)^\frac{2}{N} \rightarrow 1$ as $n \rightarrow \infty$ and $\tilde{v}_n(x):=v_n(\frac{x}{t_n}) \in \mathcal{S}'(a)$. It is easy to see that
\begin{align*}
m(a) \leq&\ I(s(\tilde{v}_n)*\tilde{v}_n)  \\
        \leq&\ I(s(\tilde{v}_n)*\tilde{v}_n)-I(s(\tilde{v}_n)*v_n)+I(s(\tilde{v}_n)*v_n)   \\
        \leq&\ |I(s(\tilde{v}_n)*\tilde{v}_n)-I(s(\tilde{v}_n)*v_n)|+I(v_n)  \\
        \leq&\ |I(s(\tilde{v}_n)*\tilde{v}_n)-I(s(\tilde{v}_n)*v_n)| +m(a_n)+\frac{1}{n}.
\end{align*}
Clearly,
\begin{align*}
    &\ |I(s(\tilde{v}_n)*\tilde{v}_n)-I(s(\tilde{v}_n)*v_n)|    \\
    =&\ \bigg|\frac{1}{2}(t^{N-2}_n-1)\int_{\mathbb{R}^N} |\nabla (s(\tilde{v}_n)*v_n)|^2 + (t^{N-2}_n-1)\int_{\mathbb{R}^N} |s(\tilde{v}_n)*v_n|^2 |\nabla (s(\tilde{v}_n)*v_n)|^2    \\
    &-\frac{1}{p}(t_n^N-1)\int_{\mathbb{R}^N}|s(\tilde{v}_n)*v_n|^p\bigg|  \\
    \leq&\ \frac{1}{2}|t^{N-2}_n-1|\int_{\mathbb{R}^N} |\nabla (s(\tilde{v}_n)*v_n)|^2 + |t^{N-2}_n-1|\int_{\mathbb{R}^N} |s(\tilde{v}_n)*v_n|^2 |\nabla (s(\tilde{v}_n)*v_n)|^2   \\
    &+\frac{1}{p}|t_n^N-1|\int_{\mathbb{R}^N}|s(\tilde{v}_n)*v_n|^p.
\end{align*}
By Lemma \ref{lem2.5}(2) and Step 1, we can conclude that $\int_{\mathbb{R}^N} |v_n|^2 |\nabla (v_n)|^2 \leq C$ and $\{v_n\}$ is bounded in $H^1(\mathbb{R}^N)$, which combining with $t_n \rightarrow 1$ as $n \rightarrow \infty$ we see that $\int_{\mathbb{R}^N} |\tilde{v}_n|^2 |\nabla (\tilde{v}_n)|^2 \leq C$ and $\{\tilde{v}_n\}$ is bounded in $H^1(\mathbb{R}^N)$.

\textbf{Claim 1:}
There exists a sequence  $\{\tilde{y}_n\} \subset \mathbb{R}^N$ and $\tilde{v} \in H^1(\mathbb{R}^N)$ such that up to a subsequence $\tilde{v}_n(\cdot +\tilde{y}_n) \rightarrow \tilde{v}\neq 0$ a.e. in $\mathbb{R}^N$.

Set
$$\tilde{\rho} :=\limsup_{n \rightarrow \infty} \sup_{\tilde{y} \in \mathbb{R}^N} \int_{B_1(\tilde{y})} |\tilde{v}_n|^2.$$
 If $\tilde{\rho}=0$, then $\tilde{v}_n \rightarrow 0$ in $L^{p}(\mathbb{R}^N)$. It is easy to see that
 $$\int_{\mathbb{R}^N}|v_n(x)|^pdx=\int_{\mathbb{R}^N}|\tilde{v}_n(t_nx)|^pdx=t_n^{-N}\int_{\mathbb{R}^N}|\tilde{v}_n(x)|^pdx \rightarrow 0.$$
At the same time, $P(v_n)=0$ gives us
$$\int_{\mathbb{R}^N} |\nabla v_n|^2 +
(N+2)\int_{\mathbb{R}^N} |v_n|^2 |\nabla v_n|^2 =\gamma_p \int_{\mathbb{R}^N} |v_n|^p \rightarrow 0.$$
In view of Lemma \ref{lem2.2}, we obtain
$$0=P(v_n) \geq \frac{1}{2} \left(\int_{\mathbb{R}^N} |\nabla v_n|^2 +\int_{\mathbb{R}^N} |v_n|^2 |\nabla v_n|^2\right),$$
this is impossible. Therefore, $\tilde{\rho}>0$, which implies that there exists a sequence  $\{\tilde{y}_n\} \subset \mathbb{R}^N$ and $\tilde{v} \in H^1(\mathbb{R}^N)$ such that up to a subsequence $\tilde{v}_n(\cdot +\tilde{y}_n) \rightarrow \tilde{v}\neq 0$ a.e. in $\mathbb{R}^N$.

\textbf{Claim 2:}
$\limsup_{n \rightarrow \infty}s(\tilde{v}_n) \leq C.$

By Claim 1, we can let $\tilde{z}_n:=\tilde{v}_n(\cdot +\tilde{y}_n) \rightarrow \tilde{v} \neq 0$. Suppose that up to a subsequence $s(\tilde{v}_n) \rightarrow +\infty$ as $n \rightarrow \infty$. Hence, it follows from Lemma \ref{lem2.4}(3) that

$$s(\tilde{z}_n)=s(\tilde{v}_n(\cdot +\tilde{y}_n))=s(\tilde{v}_n) \rightarrow + \infty.$$
As a result,
\begin{align*}
0&\leq e^{-(N+2)s(\tilde{z}_n)} I(s(\tilde{z}_n)*\tilde{z}_n) \\
&\leq  \frac{1}{2}e^{-Ns(\tilde{z}_n)} \int_{\mathbb{R}^N} |\nabla \tilde{z}_n|^2 + \int_{\mathbb{R}^N} |\tilde{z}_n|^2 |\nabla \tilde{z}_n|^2 -\frac{1}{p}e^{[p\gamma_p-(N+2)]s(\tilde{z}_n)}\int_{\mathbb{R}^N}|\tilde{z}_n|^p  \\
&\rightarrow -\infty,
\end{align*}
as $n \rightarrow \infty$, a contradiction. The proof of Claim 2 is complete.

It is easy to see that
$$\int_{\mathbb{R}^N} |\nabla (s(\tilde{v}_n)*v_n)|^2 \leq C,$$
 $$\int_{\mathbb{R}^N} |s(\tilde{v}_n)*v_n|^2 |\nabla (s(\tilde{v}_n)*v_n)|^2 \leq C,$$
 $$\int_{\mathbb{R}^N}|s(\tilde{v}_n)*v_n|^p \leq C.$$
As a result,
$$m(a) \leq m(a_n) +o_n(1),$$
to wit,
$$m(a) \leq \liminf_{n \rightarrow \infty}m(a_n).$$
\end{proof}

\begin{remark}\label{rem7.2}
    It is easy to see that the function $a \mapsto m_\mu(a)$ is lower semicontinuous with respect to any $a>0$ when $4+\frac{4}{N}<p< 2\cdot2^*$.
\end{remark}

\begin{lemma}\label{lem6.1}
 If $4+\frac{4}{N}<p< 2^*$, then $m(a) \rightarrow 0^+$ as $a \rightarrow +\infty$.
\end{lemma}

\begin{proof}
Fix $u \in \mathcal{S}'(1) \cap L^{\infty}(\mathbb{R}^N)$, set $u_a:=au \in \mathcal{S}'(a)$ for any $a>1$. By Lemma \ref{lem2.4}(1), there exists a unique $s_a \in \mathbb{R}$ such that $s_a*u_a \in \mathcal{P}'(a)$. And by means of Remark \ref{rem3.5}, we know that
\begin{align*}
0<&\; m(a)\leq I(s_a*u_a)  \\
=&\ \frac{1}{2}e^{2s_a} \int_{\mathbb{R}^N} |\nabla u_a|^2 + e^{(N+2)s_a}\int_{\mathbb{R}^N} |u_a|^2 |\nabla u_a|^2 -\frac{1}{p}e^{p\gamma_p s_a}\int_{\mathbb{R}^N}  |u_a|^p   \\
\leq&\ \frac{1}{2}a^2e^{2s_a} \int_{\mathbb{R}^N} |\nabla u|^2 + a^4e^{(N+2)s_a}\int_{\mathbb{R}^N} |u|^2 |\nabla u|^2.
\end{align*}
Next, we prove that
$$\lim_{a \rightarrow + \infty}ae^{s_a}=\lim_{a \rightarrow + \infty}ae^{\frac{N+2}{4}s_a}=0.$$
Indeed, by $s_a*u_a \in \mathcal{P}'(a)$, we see that
$$a^2e^{2s_a} \int_{\mathbb{R}^N} |\nabla u|^2 + (N+2)a^4e^{(N+2)s_a}\int_{\mathbb{R}^N} |u|^2 |\nabla u|^2 =\gamma_p a^p e^{p\gamma_p s_a}\int_{\mathbb{R}^N}  |u|^p.$$
Clearly,
\begin{align*}
&\ a^{-2}e^{-Ns_a} \int_{\mathbb{R}^N} |\nabla u|^2 + (N+2)\int_{\mathbb{R}^N} |u|^2 |\nabla u|^2  \\
=&\ \gamma_p a^{p-4} e^{[p\gamma_p-(N+2)] s_a}\int_{\mathbb{R}^N}  |u|^p.
\end{align*}
If $\limsup_{a \rightarrow + \infty}ae^{\frac{N}{2}s_a}=+\infty$ \,or\, $\limsup_{a \rightarrow + \infty}ae^{\frac{N}{2}s_a}=C_1>0$, we can derive a contradiction. As a result,
$$\lim_{a \rightarrow + \infty}ae^{\frac{N}{2}s_a}=0 \;\;\mbox{and}\;\; \lim_{a \rightarrow + \infty}e^{s_a}=0.$$
Assume $\limsup_{a \rightarrow + \infty}ae^{s_a}=C_2>0$, we can deduce that for $M>\gamma_p^{-1}C_2^{-\frac{4}{N-2}}\frac{\int_{\mathbb{R}^N} |\nabla u|^2}{\int_{\mathbb{R}^N} |u|^{2^*}}>0$, there exists a $\delta>0$ such that as $|t|\leq \delta$, we have
$$|t|^p\geq M|t|^{2^*}.$$
Therefore, for $a$ satisfying $|e^{\frac{N}{2}s_a}au|\leq \delta$, there holds
\begin{align*}
&\ a^{-N}e^{-Ns_a} \int_{\mathbb{R}^N} |\nabla u|^2 + (N+2)a^{-(N-2)}\int_{\mathbb{R}^N} |u|^2 |\nabla u|^2 \\
=&\ \gamma_p a^{-(N+2)}e^{[p\gamma_p-(N+2)-\frac{Np}{2}] s_a}\int_{\mathbb{R}^N}  |e^{\frac{N}{2}s_a}au|^p  \\
\geq&\ \gamma_p a^{2^*-(N+2)}e^{[p\gamma_p-(N+2)+\frac{N}{2}(2^*-p)] s_a} M \int_{\mathbb{R}^N}|u|^{2^*} \\
=&\ \gamma_pM(ae^{s_a})^{2^*-(N+2)}\int_{\mathbb{R}^N}|u|^{2^*},
\end{align*}
which implies
$$C_2^{-N}\int_{\mathbb{R}^N} |\nabla u|^2 \geq \gamma_p MC_2^{2^*-(N+2)}\int_{\mathbb{R}^N}|u|^{2^*}.$$
This is impossible. And if $\limsup_{a \rightarrow + \infty}ae^{s_a}=+\infty$, we can derive a contradiction. Thereby,
$$\lim_{a \rightarrow + \infty}ae^{s_a}=0.$$
And it follows from $\lim_{a \rightarrow + \infty}e^{s_a}=0$ that
$$\lim_{a \rightarrow + \infty}ae^{\frac{N+2}{4}s_a}=0.$$
It is clear that $m(a) \rightarrow 0^+$ as $a \rightarrow +\infty$.

\end{proof}

\begin{lemma}\label{lem6.2}
   If $4+\frac{4}{N}<p< 2\cdot2^*$, then $m(a) \rightarrow +\infty$ as $a \rightarrow 0^+$.
\end{lemma}
\begin{proof}
It is sufficient to show that for any sequence $\{u_n\} \subset H^1(\mathbb{R}^N)\setminus \{0\}$ such that $\lim_{n \rightarrow \infty}\|u_n\|_2=0$ and $P(u_n)=0$, one has $I(u_n) \rightarrow +\infty$ as $n \rightarrow \infty$.

Take $s_n$ such that
$$e^{-2s_n} \int_{\mathbb{R}^N} |\nabla u_n|^2 + e^{-(N+2)s_n}\int_{\mathbb{R}^N} |u_n|^2 |\nabla u_n|^2=1.$$
Set $v_n:=(-s_n)*u_n$, it is clear that
$$ \int_{\mathbb{R}^N} |\nabla v_n|^2+\int_{\mathbb{R}^N} |v_n|^2 |\nabla v_n|^2=1,\;  \int_{\mathbb{R}^N} |v_n|^2= \int_{\mathbb{R}^N} |u_n|^2 \rightarrow 0\;\; \mbox{and}\;\; s(v_n)=s_n.$$
So we obtain $\int_{\mathbb{R}^N} |v_n|^p \rightarrow 0$.
Since $P(s(v_n)*v_n)=P(u_n)=0$, we derive that for $s>0$,
\begin{align*}
    &\ I(s(v_n)*v_n) \geq I(s*v_n) \\
      =&\ \frac{1}{2}e^{2s} \int_{\mathbb{R}^N} |\nabla v_n|^2 + e^{(N+2)s}\int_{\mathbb{R}^N} |v_n|^2 |\nabla v_n|^2 -\frac{1}{p}e^{p\gamma_ps}\int_{\mathbb{R}^N}  |v_n|^p \\
      \geq&\ \frac{1}{2}e^{2s} \left(\int_{\mathbb{R}^N} |\nabla v_n|^2 +\int_{\mathbb{R}^N} |v_n|^2 |\nabla v_n|^2\right) +o_n(1) \\
      =&\ \frac{1}{2}e^{2s} +o_n(1).
\end{align*}
As $s>0$ is arbitrary, it is clear that $I(u_n)=I(s(v_n)*v_n) \rightarrow +\infty$ as $n \rightarrow \infty$.
\end{proof}

\begin{proof}[Proof of Theorem \ref{Thm1.5}]
Together with Remark \ref{rem3.5}, Lemma \ref{lem6.3}$-$Lemma \ref{lem6.2}, we can complete the proof of Theorem \ref{Thm1.5}.
\end{proof}

\section*{Acknowledgements}
This paper is partially supported by the National Natural Science Foundation of China (No. 11571200) and the Natural Science Foundation of Shandong Province (No. ZR2021MA062).

\section*{Conflict of Interest Statement}
The authors declare that they have no conflict of interest.

\section*{Data Availability Statement}
The manuscript has no associated data.


\begin{thebibliography}{99}

\bibitem{A2008}
M. Agueh, Sharp Gagliardo-Nirenberg inequalities via $p$-Laplacian type equations, NoDEA Nonlinear Differential Equations Appl., {\bf15} (2008), 457-472.

\bibitem{Bd2013}
T. Bartsch, S. de Valeriola, Normalized solutions of nonlinear Schr\"odinger equations, Arch. Math. (Basel), {\bf100} (2013), 75-83.

\bibitem{BN1990}
F.G. Bass, N.N. Nasonov, Nonlinear electromagnetic-spin waves, Phys. Rep., {\bf189} (1990), 165-223.

\bibitem{BL1983}
H. Berestycki, P.-L. Lions, Nonlinear scalar field equations, II: Existence of infinitely many solutions, Arch. Ration. Mech. Anal., {\bf82} (1983), 347-375.

\bibitem{CJ2004}
M. Colin, L. Jeanjean, Solutions for a quasilinear Schr\"odinger equation: a dual approach, Nonlinear Anal., {\bf56} (2004), 213-226.

\bibitem{CJS2010}
M. Colin, L. Jeanjean, M. Squassina, Stability and instability results for standing waves of
quasi-linear Schr\"odinger equations, Nonlinearity, {\bf23} (2010), 1353-1385.

\bibitem{GG2024}
F. Gao, Y. Guo, Existence of normalized solutions for mass super-critical
quasilinear Schr\"odinger equation with potentials, J. Geom. Anal., {\bf34} (2024), 329.

\bibitem{G1993}
N. Ghoussoub, Duality and perturbation methods in critical point theory, Cambridge Tracts in Mathematics, vol. 107, Cambridge University Press, Cambridge, 1993.

\bibitem{GZ2021}
T. Gou, Z. Zhang, Normalized solutions to the Chern-Simons-Schr\"odinger system,
J. Funct. Anal., {\bf280} (2021), 108894.

\bibitem{GY2024}
Y. Guo, Y. Yu, Multiple normalized solutions for first order Hamiltonian systems, SIAM J. Math. Anal., {\bf56} (2024), 3861-3885.

\bibitem{HanLin}
Q. Han, F.-H. Lin, Elliptic Partial Differential Equations, Courant Lecture Notes, vol. 1, American Mathematical Society, Providence, 1997.

\bibitem{H1980}
R.W. Hasse, A general method for the solution of nonlinear soliton and kink Schr\"odinger equations, Z. Phys. B, {\bf37} (1980), 83-87.

\bibitem{NI2014}
N. Ikoma, Compactness of minimizing sequences in nonlinear
Schr\"odinger systems under multiconstraint conditions,  Adv. Nonlinear Stud., {\bf14} (2014), 115-136.

\bibitem{J1997}
L. Jeanjean, Existence of solutions with prescribed norm for semilinear elliptic equations, Nonlinear Anal., {\bf28} (1997), 1633-1659.

\bibitem{JL2020}
L. Jeanjean, S.-S. Lu, A mass supercritical problem revisited, Calc. Var. Partial Differential Equations, {\bf59} (2020), 174.

\bibitem{JL2013}
L. Jeanjean, T. Luo, Sharp nonexistence results of prescribed $L^2$-norm solutions for some class of Schr\"odinger-Poisson and quasi-linear equations, Z. Angew. Math. Phys., {\bf64} (2013), 937-954.

\bibitem{JLW2015}
L. Jeanjean, T. Luo, Z.-Q. Wang, Multiple normalized solutions for quasi-linear Schr\"odinger equations, J. Differential Equations, {\bf259} (2015), 3894-3928.

\bibitem{ZZJ2025}
L. Jeanjean, J. Zhang, X. Zhong, Existence and limiting profile of energy ground
states for a quasi-linear schr\"odinger equations:
mass super-critical case, 2025, arXiv:2501.03845.

\bibitem{KIK1990}
A.M. Kosevich, B.A. Ivanov, A.S. Kovalev, Magnetic solitons, Phys. Rep., {\bf194} (1990), 117-238.

\bibitem{K1981}
S. Kurihara, Large-amplitude quasi-solitons in superfluid films, J. Phys. Soc. Japan, {\bf50} (1981), 3262-3267.

\bibitem{LZ2023}
H. Li, W. Zou, Quasilinear Schr\"odinger equations: ground state and infinitely many normalized solutions, Pacific J. Math., {\bf322} (2023), 99-138.

\bibitem{LZ2023-2}
H. Li, W. Zou, Normalized ground state for the Sobolev critical Schr\"odinger equation involving Hardy term with combined nonlinearities, Math. Nachr., {\bf296} (2023), 2440-2466.

\bibitem{LZ2024}
Q. Li, W. Zou, Normalized ground states for Sobolev critical nonlinear Schr\"odinger equation in the  $L^2$-supercritical case, Discrete Contin. Dyn. Syst., {\bf44} (2024), 205-227.

\bibitem{LL1997}
E.-H. Lieb, M. Loss, Analysis, Graduate Studies in Mathematics, vol. 14, American Mathematical
Society, Providence, 2001.

\bibitem{LS1978}
A.G. Litvak, A.M. Sergeev, One dimensional collapse of plasma waves, JETP Lett., {\bf27} (1978), 517-520.

\bibitem{LWW2003}
J.-q. Liu, Y.-q. Wang, Z.-Q. Wang, Soliton solutions for quasilinear Schr\"odinger equations II, J. Differential Equations, {\bf187} (2003), 473-493.

\bibitem{LWW2004}
J.-q. Liu, Y.-q. Wang, Z.-Q. Wang, Solutions for quasilinear Schr\"odinger equations via the
Nehari method, Comm. Partial Differential Equations, {\bf29} (2004), 879-901.

\bibitem{LLW2013}
X. Liu, J. Liu, Z.-Q. Wang, Ground states for quasilinear Schr\"odinger equations with critical growth, Calc. Var. Partial Differential Equations, {\bf46} (2013), 641-669.

\bibitem{LLW2013-2}
X.-Q. Liu, J.-Q. Liu, Z.-Q. Wang, Quasilinear elliptic equations with critical growth via perturbation method, J. Differential Equations, {\bf254} (2013), 102-124.

\bibitem{ZhangZhitao}
H. Luo, Z. Zhang, Normalized solutions to the fractional Schr\"odinger equations with combined nonlinearities, Calc. Var. Partial Differential Equations, {\bf59} (2020), 143.

\bibitem{MF1984}
V.G. Makhankov, V.K. Fedyanin, Non-linear effects in quasi-one-dimensional models of condensed matter theory, Phys. Rep., {\bf104} (1984), 1-86.

\bibitem{2024New}
A. Mao, S. Lu, Normalized solutions to the quasilinear Schr\"odinger equations with combined nonlinearities, Proc. Edinb. Math. Soc. (2), {\bf67} (2024), 349-387.

\bibitem{P1979}
R.S. Palais, The principle of symmetric criticality, Comm. Math. Phys., {\bf69} (1979), 19-30.

\bibitem{PG1976}
M. Porkolab, M.V. Goldman, Upper-hybrid solitons and oscillating-two-stream instabilities, Phys. Fluids, {\bf19} (1976), 872-881.

\bibitem{QC1982}
G.R.W. Quispel, H.W. Capel, Equation of motion for the Heisenberg spin chain, Phys. A, {\bf110} (1982), 41-80.

\bibitem{R1986}
P.H. Rabinowitz, Minimax methods in critical point theory with applications to differential equations, Regional Conference Series in Mathematics, no. 65, American Mathematical
Society, Providence, 1986.

\bibitem{N2020}
N. Soave, Normalized ground states for the NLS equation with combined nonlinearities: the Sobolev critical case, J. Funct. Anal., {\bf279} (2020), 108610.

\bibitem{S2020}
N. Soave, Normalized ground states for the NLS equation with combined nonlinearities, J. Differential Equations, {\bf269} (2020), 6941-6987.

\bibitem{Willem}
M. Willem, Minimax Theorems, Progress in nonlinear differential equations and
their applications, vol. 24, Birkh\"auser, Basel, 1996.

\bibitem{ZZ2018}
X. Zeng, Y. Zhang, Existence and asymptotic behavior for the ground state of quasilinear elliptic equations, Adv. Nonlinear Stud., {\bf18} (2018), 725-744.

\bibitem{ZCW2023}
L. Zhang, J. Chen, Z.-Q. Wang, Ground states for a quasilinear Schr\"odinger equation: Mass critical
and supercritical cases, Applied Mathematics Letters, {\bf145} (2023), 108763.

\bibitem{ZZ2022}
Z. Zhang, Z. Zhang, Normalized solutions to $p$-Laplacian equations with combined nonlinearities, Nonlinearity, {\bf35} (2022), 5621-5663.


\end{thebibliography}
\end{document}